\documentclass[10pt]{amsart}

\usepackage{graphicx}
\usepackage{amsmath,amssymb,amsthm,amsfonts,dsfont,mathrsfs,stmaryrd,float}
\usepackage[latin1]{inputenc}
\usepackage[T1]{fontenc}
\usepackage{enumitem}
\usepackage[style=alphabetic,natbib=true,backend=biber]{biblatex}
\usepackage{hyperref}
\usepackage{color}
\usepackage[usenames,dvipsnames]{xcolor}
\usepackage{bm}
\usepackage[cal=boondoxo]{mathalfa}
\usepackage[normalem]{ulem} 
\usepackage{colonequals}

\allowdisplaybreaks[4]

\newtheorem{thm}{Theorem}[section]
\newtheorem*{thm*}{Theorem}
\newtheorem{lem}[thm]{Lemma}

\newtheorem{cor}[thm]{Corollary}
\newtheorem{defn}[thm]{Definition}
\newtheorem{rem}[thm]{Remark}
\newcommand{\A}{\mathcal{A}}

\newcommand{\F}{\mathfrak{F}}

\newcommand{\R}{\mathbb{R}}
\newcommand{\Z}{\mathbb{Z}}
\newcommand{\C}{\mathbb{C}}
\newcommand{\N}{\mathbb{N}}
\renewcommand{\S}{\mathcal{S}}
\newcommand{\T}{\mathcal{T}}
\renewcommand{\SS}{\mathbb{S}}

\renewcommand{\Re}{\mathop{\mathrm{Re}}}

\newcommand{\Orb}{\mathop{\mathrm{Orb}}}

\renewcommand{\H}{\mathcal{H}}

\newcommand{\z}{\mathcal{z}}

\newcommand{\CC}{\mathcal{C}}

\newcommand{\bA}{\bar{\mathcal{A}}}

\bibliography{bibliography.bib}

\begin{document}

\title{Characterization of minimal sequences associated with self-similar interval exchange maps}

\author{Milton Cobo}
\address{Departamento de Matem\'atica, Universidade Federal do Esp\'{\i}rito Santo, Av. Fernando Ferrari 514, Goiabeiras, Vit\'oria, Brasil.} \email{milton.e.cobo@gmail.com}

\author{Rodolfo Guti\'errez-Romo}
\address{Departamento de Ingenier\'{\i}a
Matem\'atica and Centro de Modelamiento Ma\-te\-m\'a\-ti\-co, CNRS-UMI 2807, Universidad de Chile, Beauchef 851, Santiago,
Chile.}\email{rgutierrez@dim.uchile.cl}

\author{Alejandro Maass}
\address{Departamento de Ingenier\'{\i}a
Matem\'atica and Centro de Modelamiento Ma\-te\-m\'a\-ti\-co, CNRS-UMI 2807, Universidad de Chile, Beauchef 851, Santiago,
Chile.}\email{amaass@dim.uchile.cl}

\begin{abstract}
The construction of affine interval exchange maps with wandering intervals that are semi-conjugate to a given self-similar interval exchange map is strongly related to the existence of the so called minimal sequences associated with local potentials, which are certain elements of the substitution subshift arising from the given interval exchange map. In this article, under the condition called unique representation property, we characterize such minimal sequences for potentials coming from non-real eigenvalues of the substitution matrix.
We also give conditions on the slopes of the affine extensions of a self-similar interval exchange map that determine whether it exhibits a wandering interval or not.
\end{abstract}

\maketitle
\markboth{Milton Cobo, Rodolfo Guti\'errez-Romo, Alejandro
Maass}{Characterization of minimal sequences associated with self-similar i.e.m.}

\section{Introduction}\label{sec:introduction}
Let $I=[0,1)$ and $\A$ be a finite alphabet. 
An \emph{interval exchange map} (i.e.m.)\ is a bijective map $T\colon I\to I$ such that for a partition of 
$I$ by intervals $(I_a; a \in \A)$ of length $(|I_a|; a \in \A)$ and a vector $\delta=(\delta_a;a\in \A) \in\R^\A$ we have $T(t)= t+\delta_a$ if $t\in I_a$. In a similar way, an \emph{affine interval exchange map} (affine i.e.m.)\ is a bijective map 
$f\colon I\to I$ such that for a vector $\ell=(\ell_a; a\in \A) \in \R^\A$ with positive entries and a vector $\mathcal{d}=(\mathcal{d}_a; a\in \A) \in \R^\A$ we have $f(t)= \ell_a t + \mathcal{d}_a$ if $t\in I_a$. The vector $\ell$ is called the \emph{slope vector} of $f$.
An i.e.m.\ $T$ is self-similar if the first return map of $T$ to a proper interval $[0,\alpha) \subseteq I$ is, up to rescaling, equal to $T$. 

An affine i.e.m.\ $f$ is \emph{semi-conjugate} to an i.e.m.\ $T$ if there exists a continuous, surjective and non-decreasing map $ h\colon I \to I$ such that $h \circ f = T\circ h$. 
We refer to $f$ as an \textit{affine extension of $T$}. 
It can be visualized as an ``affine perturbation'' of $T$ in the sense that it can be obtained from the graph of $T$ by perturbing the slopes.

Given an i.e.m.\ $T$, the existence of an affine i.e.m.\ $f$ semi-conjugate to $T$ and having wandering intervals has been studied in several works during the last twenty years. As was established in \cite{yoccoz-wandering}, this situation is generic in the space of parameters describing interval exchange maps, although there are some restrictions on the possible slope vectors $\ell$ ($\log\ell=(\log\ell_a;a\in \A)$ should expand at a rate given by the second largest Lyapunov exponent of the associated Rauzy--Veech--Zorich cocycle).

Aside from \cite{yoccoz-wandering}, most results concern self-similar i.e.m.'s. In this context, 
the pioneering work \cite{cam-gut} established that there exists an affine i.e.m.\ with slope vector $\ell$ that is semi-conjugate to the i.e.m.\ $T$ if and only if the length vector $\lambda = (|I_a|; a \in \A)$ of $T$ is orthogonal to 
$\log\ell$. This condition is also equivalent to the fact that $\log\ell$ is generated by eigenvectors different from the Perron--Frobenius eigenvector of the matrix $M$ associated with $T$. In subsequent works \cite{cam-gut}, \cite{cobo}, \cite{persistence}, \cite{bressaud-deviation} and \cite{CGM}, the existence of such an affine i.e.m.\ having wandering intervals is shown to be related to the spectral properties of $M$. 

More precisely, it was proved in \cite{persistence} that, if $M$ admits a real eigenvalue $\beta > 1$ 
different from the Perron--Frobenius eigenvalue, but Galois conjugate to it, 
and an associated eigenvector $\gamma=(\gamma_a; a\in \A)$, then there exists an affine i.e.m.\ with slope vector $\ell=(\exp(-\gamma_a);a\in \A)$ which is semi-conjugate to $T$ and has wandering intervals. Then, in \cite{CGM} this result was extended to the case of a non-real eigenvalue $\beta$ with $|\beta|>1$ such that $\beta/|\beta|$ is not a root of unity and under the so called unique representation property for $\beta$ and some associated eigenvector $\Gamma$ (see Definition \ref{def:uniquerep}). More precisely, it was proved that for almost all complex 
eigenvectors $\gamma$ in the complex vector space generated by $\Gamma$ there exists an affine i.e.m.\ with slope vector 
$(\exp(-\Re(\gamma_a)); a\in \A)$ that is semi-conjugate to $T$ and has a wandering interval. 
When $\log\ell$ is an eigenvector of $M$ associated with $\beta$ and $|\beta| \leq 1$, then 
any affine i.e.m.\ which is semi-conjugate to $T$ is indeed conjugate and, therefore, does not have wandering intervals (see \cite{cam-gut}, \cite{cobo} and \cite{bressaud-deviation}).

The construction of affine perturbations of the i.e.m.\ $T$ with wandering intervals is somehow difficult. 
The strategy followed in \cite{persistence} and \cite{CGM}, when $T$ is self-similar, is the one proposed by Camelier and Guti\'errez in \cite{cam-gut}.  
It consists of ``blowing up'', \emph{\`a la Denjoy}, the orbits of specific points of the interval called \emph{distinguished points} of a complex vector $\gamma \in \C^\A$. The set of orbits of distinguished points is finite for each $\gamma$ (see \cite{cobothesis} and \cite{yoccoz-wandering}). These points 
are intimately related to the so called \emph{minimal sequences}
of the substitutive subshift $\Omega_\sigma$ associated with the self-similar i.e.m.\ $T$. More precisely, given $\omega \in \Omega_\sigma$ and a complex vector $\gamma \in \C^\A$, define $\gamma_n(\omega) = \gamma_{\omega_0} + \dotsb + \gamma_{\omega_{n-1}}$ and $\gamma_{-n}(\omega) = -(\gamma_{\omega_{-n}} + \dotsb + \gamma_{\omega_{-1}})$ for $n \geq 0$. The sequence $\omega$ is said to be \emph{minimal for $\gamma$} if 
$\Re(\gamma_n(\omega)) \geq 0$ for all $n \in \Z$. Then, the main technical step in the strategy devised in \cite{cam-gut} requires a sequence $\omega \in \Omega_\sigma$ for a complex vector $\gamma \in \C^\A$ such that the series $\sum_{n \in \Z} \exp(-\Re(\gamma_n(\omega)))$ is convergent. A necessary condition is that $\omega$ is minimal (up to a shift) for $\gamma$. 
Conversely, the main result \cite{CGM} states conditions under which minimal sequences always correspond to itineraries of distinguished points with respect to $(I_a; a \in \A)$. In the aforementioned works, even if minimal sequences are constructed from eigenvectors $\gamma$ associated with particular choices of eigenvalues $\beta$ as described before, very little is known about their nature, besides their existence. 

We think that minimal sequences are interesting in their own. In particular, they are related to the extreme points in the boundary of some ``fractal'' sets associated with the expansive eigenspaces of integer matrices arising from substitutions. These fractals sets were first introduced by Dumont and Thomas in \cite{dumont-thomas} to study numeration systems associated with substitutions. They are, in some sense, ``dual'' to the classical Rauzy fractals, which are associated with the contractive eigenspaces substitution matrices. They are also studied in \cite{geometricalmodels} in the particular case of the cubic Arnoux--Yoccoz map. Besides these two works, very little is known about them and understanding minimal sequences can shed light on such fractal sets.

In this article we characterize the set of minimal sequences and provide a method to compute them assuming the same hypotheses as in \cite{CGM}. That is, $\beta$ is a non-real eigenvalue of the matrix associated with the self-similar i.e.m.\ $T$ with $|\beta|>1$ such that $\beta/|\beta|$ is not a root of unity and under the unique representation property for $\beta$ and some associated eigenvector $\Gamma$. 
In Theorems \ref{thm:enter minimal} and \ref{thm:minimal points}, we state that minimal sequences can be obtained iterating another map $H$ that turns out to be conjugate to an i.e.m.\ in its minimal components. This map is in fact the main novelty of this article that we think can play an interesting role in studying minimal sequences associated with general substitutive subshifts or can be extended to study minimal sequences associated with more general minimal subshifts, such as linearly recurrent systems. Corollary \ref{cor:finite min points} states that for almost every eigenvector $\gamma$ in the complex vector space generated by $\Gamma$ the set of minimal sequences is finite. 
Finally, in Theorems~\ref{thm:no-wandering} and \ref{thm:no_conjugacy}, we give conditions on the eigenvector $\gamma$ in the complex space generated by $\Gamma$ that determine whether the affine extensions of $T$ with slope vector $\ell=\exp(-\Re(\gamma))$ have a wandering interval or not. 

We illustrate our results in the cubic Arnoux--Yoccoz map. We prove that for almost every complex eigenvector as before there are exactly two orbits of minimal sequences (and therefore exactly two orbits of distinguished points). Interestingly, the construction along the different minimal sequences produce different affine i.e.m.'s with different wandering intervals but the same slope vector. This shows that more than one affine i.e.m.\ with the same slope vector can be semi-conjugate to the cubic Arnoux--Yoccoz map. Since it was remarked at the end of Section 3.7.2 of \cite{yoccoz-wandering} that almost all i.e.m.'s are expected to have only one orbit of distinguished points, to our knowledge such examples are new.

In the next section we introduce the necessary background. We state our main results in Section \ref{sec:the main results}. Section \ref{sec:Extreme points} is devoted to presenting the main technical consequences of our hypotheses. The map allowing to characterize minimal sequences is defined in Section \ref{sec:the map H} together with its main properties. Finally, our main results are proved in Section \ref{sec:proofs}. To illustrate our main results, the cubic Arnoux--Yoccoz map is studied in Section \ref{sec:example} and an associated Appendix. 

\section{Background and Preliminaries}\label{sec:preliminaries}

\subsection{Self-similar i.e.m.'s}\label{sec:basic facts}
Let $\A$ be a finite alphabet and 
$T\colon [0,1)\mapsto [0,1) $ be an i.e.m.\ exchanging the intervals of the partition $(I_a; a\in\A)$ of $[0, 1)$, \emph{i.e.}, $T(t) = t+\delta_a$ if $t\in I_a$, where $\delta\in\R^\A$ is the translation vector. We suppose that $T$ is \emph{self-similar} on the interval $[0,\alpha)$ with 
$\alpha<1$. That is, up to rescaling, the induced map on $[0,\alpha)$ is equal to $T$. 
Then $T$ is uniquely ergodic and minimal (every orbit of $T$ is dense in $[0,1)$).
For more details on uniquely ergodic i.e.m.'s see \cite{veech78}.

For each $a \in \A$, define the interval $I_a^{(1)} = \alpha I_a\subseteq [0,\alpha)$ and denote by $R$ the \emph{renormalization matrix} given by $R_{a,b} = |\{ 0 \leq k \leq r_b - 1; T^k(I^{(1)}_b) \subseteq I_a\}|$, where $r_b$ is the first return time of $I^{(1)}_b$ to $[0, \alpha)$. By the minimality of $T$, some power of $R$ is a positive matrix. We have that $\alpha^{-1} > 1$ is its Perron--Frobenius eigenvalue and it is easy to see that the vector of lengths 
$\lambda = (|I^{(1)}_a|;a\in\A)$ is an eigenvector of $R$ associated with $\alpha^{-1}$. Also, the translation vector $\delta$ is an eigenvector of the transpose matrix $M=R^t$ associated with $\alpha$. 

\subsection{Substitution subshifts and minimal sequences}
Let $\A$ be a finite set or alphabet and let $\A^*$ be the set of all words in $\A$. For $w\in\A^*$, $|w|$ denotes its length, \emph{i.e.}, the number of letters in $w$.
The empty word is denoted by $\varepsilon$.

A substitution is a map $\sigma \colon \A \to \A^*\setminus \{\varepsilon\}$. 
It naturally extends to the set of two-sided sequences $\A^\Z$ by concatenation. That is, for
$\omega=(\omega_m)_{m\in \Z} \in \A^{\Z }$ the extension is given by 
$$
\sigma(\omega)=\ldots{} \sigma(\omega_{-2})\sigma(\omega_{-1}) \cdot \sigma(\omega_0)\sigma(\omega_1) \ldots{} , 
$$
where the central dot separates negative and nonnegative coordinates. 
A further natural convention is that $\sigma(\varepsilon) = \varepsilon$.

We define the matrix $M^\sigma$ associated with $\sigma$ by: the entry $M^\sigma_{a,b}$ is the number of times the letter $b$ appears in $\sigma(a)$ for any $a,b \in \A$. 
The substitution is said to be primitive if a power of $M^{\sigma}$ is strictly positive. 
This means that for some $n \geq 0$ any letter in $\A$ appears in the $n$-th iterate of the substitution of any other letter in $\A$.

Let $\Omega_\sigma \subseteq \A^\Z$ be the subshift defined from $\sigma$. That is, $\omega \in \Omega_\sigma$ if and only if any subword of $\omega$ is a subword of $\sigma^n(a)$ for some integer $n \geq 0$ and $a \in \A$. We denote by $ S $ the left shift in $\Omega_\sigma$.
We call $(\Omega_\sigma, S)$ the substitution subshift associated with $\sigma$. This subshift is minimal whenever $\sigma$ is primitive.

If $\sigma$ is primitive, by the recognizability property (see \cite{mosse}), given a point $\omega \in \Omega_\sigma$ there exists a unique sequence $(p_m,c_m,s_m)_{m\geq 0} \in (\A^* \times \A \times \A^*)^\N$ such that for each integer $m \geq 0$ we have $\sigma(c_{m+1})=p_m c_m s_m$ and
$$
{}\ldots \sigma^3(p_3) \sigma^2(p_2)\sigma^1(p_1) p_0 \cdot c_0 s_0 \sigma^1(s_1)\sigma^2(s_2)\sigma^3(s_3)\ldots{}
$$
is the central part of $\omega$, where the dot separates negative and nonnegative coordinates.
This sequence is called the prefix-suffix decomposition of $\omega$.

We refer to \cite{queff} and \cite{fogg} and references therein for the general theory of substitutions.

To a self-similar i.e.m.\ $T$ we associate a substitution subshift in the following way.
Given $t \in [0,1)$ we construct a symbolic sequence 
$\omega=(\omega_m)_{m\in \Z} \in \A^\Z$ by the rule 
$\omega_m=a$ if and only if $T^m(t) \in I_a$. The sequence $\omega$ is called the itinerary of $t$.
Let $\Omega_T \subseteq A^\Z$ be the closure of the set of sequences constructed in this way for every $t \in [0,1)$. Clearly the sequence associated with $T(t)$ corresponds to $S(\omega)$, where 
$S \colon\A^\Z\to \A^\Z$ is the left shift map. Moreover, it is classical that there exists 
a continuous and surjective map $\pi_T \colon \Omega_T \to [0,1)$ such that $T\circ \pi_T=\pi_T\circ S$. The map $\pi_T$ is invertible up to a countable set of points corresponding to the orbits of discontinuities of $T$. 
Since $T$ is self-similar, the restriction of $S$ to $\Omega_T$ is minimal and $\Omega_T$ is a subshift associated with a substitution.
The substitution is constructed in the following way: $\sigma(a)=w_0\ldots w_{r_a-1}$ if and only if $T^m(I^{(1)}_a) \subseteq I_{w_m}$ for every integer $0\leq m \leq r_a - 1$ and $a \in \A$. We then have that the matrix of the substitution $M^\sigma$ is the transpose of the renormalization matrix $R$ associated with $T$, \emph{i.e.}, $M^\sigma=M=R^t$. Furthermore, $\Omega_T=\Omega_\sigma$ and $\sigma$ is primitive. For details see \cite{cam-gut}. 

\subsection{\texorpdfstring{Minimal sequences for a vector $\bm{\gamma}$}{Minimal sequences for a vector gamma}}\label{sec:minimal sequences}

Let $\sigma$ be a primitive substitution in the alphabet $\A$. 
Given a vector $\gamma=(\gamma_a;a\in\A) \in\C^\A$ and a word $w = w_0 \ldots w_{n-1} \in \A^*$ we define $\gamma(w) = \gamma_{w_0} + \ldots \gamma_{w_{n-1}}$. For a sequence $\omega = (\omega_m)_{m\in\Z} \in \Omega_\sigma$ we define $\gamma_0(\omega) = 0$, 
$\gamma_n(\omega) = \gamma( \omega_0 \ldots \omega_{n - 1})$ and
$\gamma_{-n}(\omega) = -\gamma(\omega_{-n}\ldots\omega_{-1})$ for $n \geq 1$. It is easy to see that if $\gamma$ is an eigenvector of $M^\sigma$ associated with $\beta \in \C$, then for any integer $n\geq 0$, 
\begin{equation}\label{eq:gamma(sigma^n)} 
\gamma(\sigma^n(u)) = \beta^n \gamma(u).
\end{equation}

\begin{defn}\label{def:minimalsequences}
A sequence $\omega \in \Omega_\sigma$ is a \emph{minimal sequence} for the vector 
$\gamma \in \C^{\A}$ if $$\Re( \gamma_n(\omega))\geq 0 \text{ for all } n\in \Z.$$
\end{defn}

Assume $\sigma$ is the substitution associated with a self-similar i.e.m.\ $T$. We adopt all notations of Section \ref{sec:basic facts} and we assume $\beta$ is a non-real eigenvalue of $M =M^\sigma$ with $|\beta|>1$. 
In Lemma 4.4 in \cite{CGM} it is proved that:
\begin{lem}\label{lem:existence of minimal points}
For any eigenvector $\gamma$ of $M$ associated with $\beta$ there exist minimal sequences. 
\end{lem}

\subsection{Fractals associated with a self-similar i.e.m.\ and the unique representation property}\label{subsec: fractals associated with sigma and beta}
Let $T$ be a self-similar i.e.m.\ which is self-similar on the interval $[0,\alpha)$, and $\beta$ be an eigenvalue of $M$ with $|\beta| > 1$ such that $\beta/|\beta|$ is not a root of unity. 
Consider an eigenvector $\gamma$ of $M$ for $\beta$. Recall $\sigma$ is the substitution associated with $T$. 

Denote by $\bA_a$ the set of possible triples $(p,c,s) $ in $\A^* \times \A \times \A^*$ such that
$\sigma(a)=pcs$. Set $\bA = \bigcup_{a\in\A} \bA_a$, which we call the set of \emph{labels}. We define
$$\S = \{(p_m, c_m, s_m)_{m \geq 1}\in\bA^\N; (p_{m+1}, c_{m+1}, s_{m+1})\in\bA_{c_{m}}, m\geq 1 \}$$
and, for $a \in \A$,
$$\S_a = \{(p_m, c_m, s_m)_{m \geq 1}\in\S; (p_1,c_1,s_1)\in\bA_a \}.$$ 

For $x \in \S_a$ write $x = (p_m^x, c_m^x, s_m^x)_{m \geq 1}$.
We also consider the partition of $\S_a$ by the subsets 
$$\S_{a,(p,c,s)} = \{x\in\S_a; (p^x_1,c^x_1,s^x_1) = (p,c,s)\}$$
for $(p,c,s)\in\bA_a$. 

The next concepts depend on the choice of $\gamma$, nevertheless, to simplify notations we omit this dependence. 

For each $a\in\A$ and $n\geq 1$ we define maps $\z_a \colon \S_a \mapsto \C$ and 
$\z_a^{(n)}\colon \S_a \mapsto \C$ by
$$\z_a(x) = \sum_{m \geq 1} \beta^{-m} \gamma(p_m^x) \textrm{ and }\z^{(n)}_a(x) = \sum_{m = 1}^n \beta^{-m} \gamma(p_m^x).$$

Given $ a\in\A $ and $(p,c,s)\in\bA_a$ we consider the sets (referred hereafter as 
``the fractals'')
$$\F_a = \{ \z_a(x); x \in \S_a \} \quad\text{and}\quad \F_{a,(p,c,s)} = \{ \z_a(x); x \in \S_{a,(p,c,s)} \}.$$ 
We easily notice the decomposition: $\F_a = \bigcup_{(p,c,s)\in\bA_a} \F_{a,(p,c,s)}$.
We say that a sequence $x \in \S_a$ is a representation of $z \in \F_a$ if $\z_a(x)=z$.
We also consider $\F^{(n)}_{a} = \{ \z^{(n)}_a(x); x \in \S_a \}$ and $\F^{(n)}_{a,(p,c,s)} = \{ \z^{(n)}_a(x); x \in \S_{a,(p,c,s)} \}$for each $n\geq 1$.

\begin{defn}\label{def:extremepoints}
For each $ a\in\A $ we define the maps
$$v_a(\tau) = \min_{z \in \F_a} \Re(\tau z) \quad\text{and}\quad 
v_{a,(p,c,s)}(\tau) = \min_{z \in \F_{a,(p,c,s)}} \Re(\tau z),$$
where $\tau \in \SS^1$. Analogously, for $n \geq 1$,
$$v_a^{(n)}(\tau) = \min_{z \in \F^{(n)}_a} \Re(\tau z) \quad\text{and}\quad 
v_{a,(p,c,s)}^{(n)}(\tau) = \min_{z \in \F^{(n)}_{a,(p,c,s)}} \Re(\tau z).$$
\end{defn}
Notice that $v_{a,(p,c,s)}(\tau) \geq v_a(\tau)$ and $v_a(\tau) = \min_{(p,c,s) \in \bA_a} v_{a,(p,c,s)}(\tau)$. 

In Lemmas 5.3 and 7.2 in \cite{CGM} it is proved that:
\begin{lem}
\label{lem:continuity of v_a}
For every $a \in \A$, $v_a \colon \SS^1 \to \R$ is continuous and has lateral derivatives at each $\tau\in\SS^1$.
\end{lem}

A point $z\in\F_a$ is called an extreme point for the direction 
$\tau\in\SS^1 $ if $v_a(\tau)=\Re(\tau z)$. The set of extreme points for the direction $\tau$ is written as $E_a(\tau)$. We also set $E_{a,(p,c,s)}(\tau) = E_a(\tau) \cap \F_{a, (p,c,s)}$.

The results of this article depend on the following hypothesis:

\begin{defn}\label{def:uniquerep}
We say that $T$ satisfies the unique representation property (u.r.p.)\ for $\beta$ and the eigenvector $\gamma$ if every extreme point of the associated fractals has a unique representation. 
That is, for any $a \in \A$ and any extreme point $z\in \F_a$ there exists a unique $x\in \S_a$ with $\z_a(x)=z$.

\end{defn}

The unique representation property implies in particular that extreme points of $\F_a$ do not belong to intersections $\F_{a,(p,c,s)}\cap \F_{a, (\bar p,\bar c,\bar s)}$ for distinct $(p,c,s), (\bar p,\bar c,\bar s) \in \bA_a$. In \cite{CGM} it is proved that this property holds for the cubic Arnoux--Yoccoz map.

\begin{figure}
\centering
\includegraphics[width=0.30\textwidth]{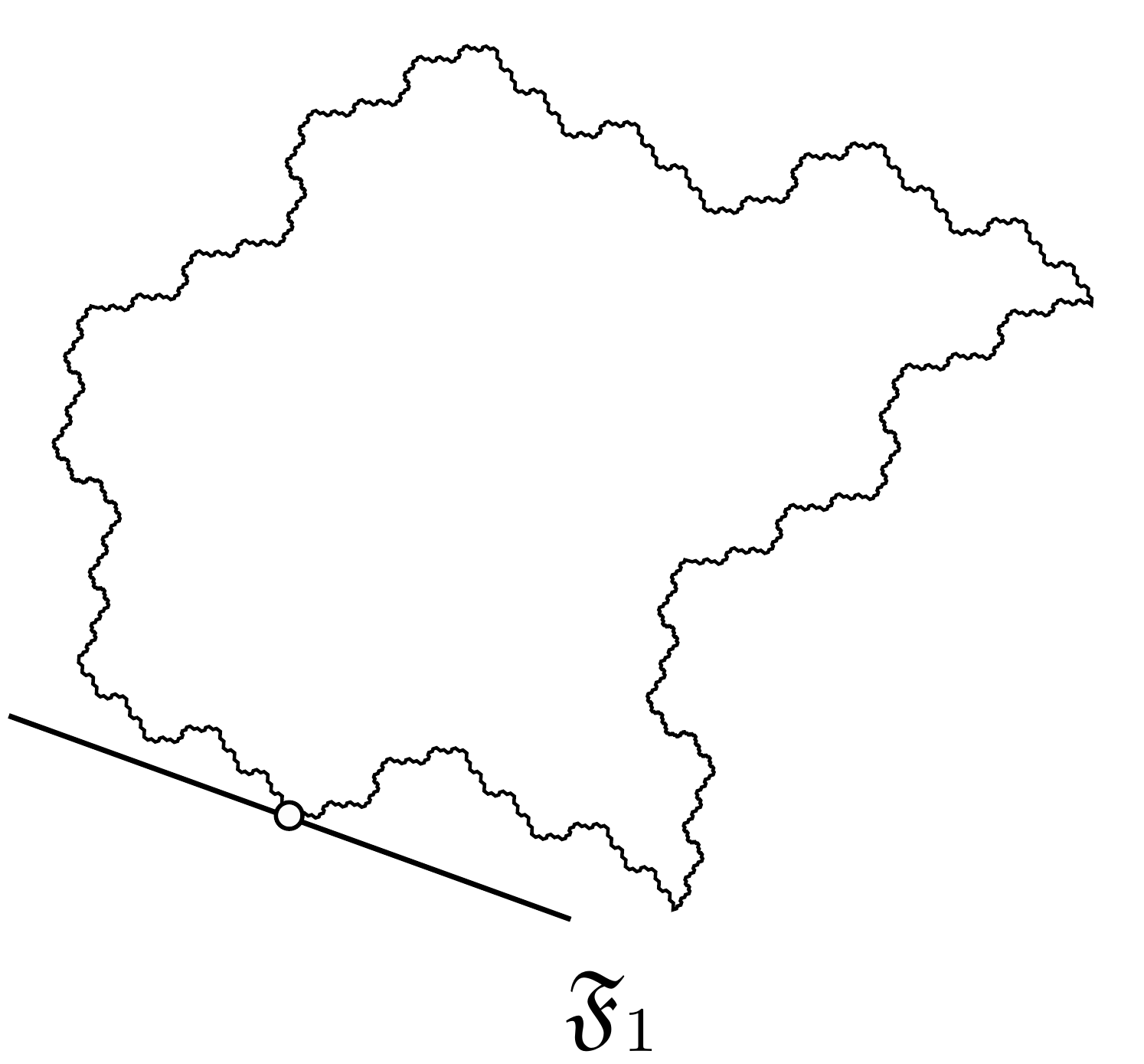}
\caption{The dot shows an extreme point of one of the Arnoux--Yoccoz fractals in the direction of the line (see Section \ref{sec:example}).}
\label{fig:extreme}
\end{figure}

\section{The main results}\label{sec:the main results}

Throughout the article we will assume the conditions of Section~\ref{sec:preliminaries}. Namely, $T$ is a self-similar i.e.m.\ on the interval $[0,\alpha)$, where $\alpha<1$. Recall that this implies that the associated symbolic system is generated by a substitution $\sigma$. We denote its matrix by $M$. We fix a non-real eigenvalue $\beta$ of $M$ satisfying that $|\beta| > 1$ and that $\beta_0 = \beta/|\beta|$ is not a root of unity, and an eigenvector $\Gamma$ for $\beta$. We assume that $T$ satisfies the u.r.p.\ for $\beta$ and $\Gamma$.

For $a\in\A$, we denote by $\Psi_a$ the set of directions $\tau \in \SS^1$ such that, for different
labels $(p,c,s)$ and $(\bar p,\bar c,\bar s)$ in $\bA_a$, we have $v_a(\tau)=v_{a,(p,c,s)}(\tau)=v_{a,(\bar p,\bar c,\bar s)}(\tau)$. This definition is slightly different from the one given in \cite{CGM}, but the two definitions coincide under the u.r.p. It will be proved in Lemma \ref{lem:psi finito} that the u.r.p.\ implies that $\Psi_a$ is finite for all $a \in \A$. 
We set $\Psi=\bigcup_{a\in\A} \Psi_a$.

Our main results depend on the following map.
Set $X = \SS^1\times \bA$ and let $H\colon X \to X$ be defined as
$H(\tau, (q, a, r)) = (\beta_0^{-1} \tau, (p, c, s))$ if $v_a(\tau)=v_{a,(p,c,s)}(\tau)$.
Clearly, the definition of $H$ is ambiguous if $\tau \in \Psi_a$, however, under the u.r.p., $H$ is well-defined and continuous except at finitely many points. The map $H$ will be extensively discussed in Section \ref{sec:the map H}. 
In particular, we will show that $H$ is conjugate to an \emph{interval translation map}, \emph{i.e.}, a piecewise isometry of an interval. These maps are different from interval exchange maps because they need not be bijective.

Our first result will be proved at the end of Section~\ref{sec:the map H}:

\begin{thm}\label{thm:H_minimal}
	Let $T$ be a self-similar i.e.m. Assume that $M$ has an eigenvalue $\beta$ with $|\beta|>1$ such that $\beta_0=\beta/|\beta|$ is not a root of unity, and that there exists an eigenvector $\Gamma$ for $\beta$ such that $T$ has the u.r.p.\ for $\beta$ and $\Gamma$. Then $H$ has finitely many minimal components and its restriction to each minimal component is an i.e.m.
\end{thm}

We need the following definitions to state our results. In what follows $\llbracket \tau - \tau' \rrbracket$ is the natural distance between $\tau$ and $\tau'$ in $\SS^1$.

\begin{defn}\label{def:good} \
	\begin{enumerate}[label=(\roman{*})]
		\item A direction $\tau \in \SS^1$ is \emph{good} if for some constant $1 < A < |\beta|$ and every $\xi \in \Psi$ we have that $\liminf_{n \to \infty} A^n \llbracket \xi - \beta_0^n \tau\rrbracket = \infty$. 
		\item Conversely, a direction $\tau \in \SS^1$ is \emph{very bad} if for every $\xi \in \Psi$ we have $\limsup_{n \to \infty} |\beta|^n \llbracket \xi - \beta_0^n \tau\rrbracket = 0$.
		\item An eigenvector $\gamma$ of $M$ for $\beta$ is said to be good (resp.\ very bad) if $\gamma=z\Gamma$, where $z/|z|\in\SS^1$ is a good (resp.\ very bad) direction.
	\end{enumerate}
\end{defn}

Notice that if $\tau$ is a good direction then necessarily
$\liminf_{n \to \infty} A^n \llbracket \xi - \beta_0^n \tau\rrbracket =\infty$
for all $A \geq |\beta|$.

The definition of good direction is different from that of \cite{CGM}, but each good direction as in Definition \ref{def:good} is also a good direction in the sense of that article. Lemma 7.4 in \cite{CGM} can be then applied to show that the set of good directions has total Lebesgue measure, which implies that the set of very bad directions has measure zero. 

We will show in Lemma \ref{lem:cota_minimos} that the main result of \cite{CGM}, Theorem 7.1, still holds for good directions in the sense of Definition \ref{def:good}.

Our next two results characterize minimal sequences for good eigenvectors. Roughly speaking, the prefix-suffix decompositions of minimal sequences for these eigenvectors are given, up to a finite number of coordinates, by the pre-orbits of $H$ in its minimal components. 

\begin{thm}\label{thm:enter minimal}
	Assume the same hypotheses of Theorem \ref{thm:H_minimal}. If $\gamma=z\Gamma$ is a good eigenvector, $\tau=z/|z| \in \SS^1$ and $\omega \in \Omega_\sigma$ is a minimal sequence for $\gamma$, then the prefix-suffix decomposition $(p_m, c_m, s_m)_{m\geq 0}$ of $\omega$ satisfies that:
	\begin{itemize}
		\item $(\beta_0^{n}\tau, (p_{n}, c_{n}, s_{n}))$ belongs to a minimal component of $H$ for some $n \geq 0$,
		\item $(\beta_0^{m}\tau, (p_m, c_m, s_m)) = H^{-m+n}(\beta_0^{n}\tau, (p_{n}, c_{n}, s_{n}))$, for all $m \geq n$.
	\end{itemize}
\end{thm}
\smallskip

Conversely, we have a way to construct minimal sequences for good eigenvectors. 

\begin{thm}\label{thm:minimal points}
	Assume the same hypotheses of Theorem \ref{thm:H_minimal}. If $\gamma=z\Gamma$ is a good eigenvector, $\tau=z/|z| \in \SS^1$ and $(\tau,(p_0,c_0,s_0))$ belongs to a minimal component of the map $H$, then, setting $H^{-m}(\tau,(p_0,c_0,s_0)) = (\beta_0^m\tau, (p_m,c_m,s_m))$ for $m\geq 0$, we have that
	$(p_m,c_m,s_m)_{m\geq 0}$ is the prefix-suffix decomposition of some shift of a minimal sequence for the vector $\gamma$. 		
\end{thm}

The fact that $H$ has finitely many minimal components allows us to deduce:

\begin{cor}\label{cor:finite min points}
	Assume the same hypotheses of Theorem \ref{thm:H_minimal}. If
	$\gamma$ is a good eigenvector, then the set of minimal sequences for $\gamma$ is finite.
\end{cor}

Depending on the logarithm of the slope vector of an affine extension of $T$, the existence or absence of wandering intervals is ensured:

\begin{thm}\label{thm:no_conjugacy}
	Assume the same hypotheses of Theorem \ref{thm:H_minimal}. If $\gamma$ is a good eigenvector, then no affine extension of $T$ with slope vector $\exp(-\Re(\gamma))$ is conjugate to~$T$. 
\end{thm}

\begin{thm}\label{thm:no-wandering}
	Assume the same hypotheses of Theorem \ref{thm:H_minimal}. If $\gamma$ is a very bad eigenvector, then every affine extension of $T$ with slope vector $\exp(-\Re(\gamma))$ is conjugate to $T$. 
\end{thm}
In other words, affine extensions constructed from good eigenvectors exhibit wandering intervals, whereas those constructed from very bad eigenvectors do not.

An important consequence of the previous results is that one can explicitly describe minimal sequences producing affine extensions with wandering intervals of a given self-similar i.e.m.\ and, thus, 
construct good approximations of such extensions. 

In particular, we apply our results to the cubic Arnoux--Yoccoz map. For this example there exists an eigenvalue $\beta$ of the associated matrix $M$ with $|\beta|>1$ and multiplicity one. 
It was proved in \cite{CGM} that the u.r.p.\ holds. We prove that that map $H$ has exactly two minimal components and, thus:

\begin{thm}
	In the cubic Arnoux--Yoccoz map, for each good eigenvector $\gamma$ associated with $\beta$, there are exactly two orbits of minimal sequences. 
\end{thm}

By the construction shown at the beginning of Section 2 of  \cite{CGM}, to each of these two orbits of minimal sequences corresponds an affine i.e.m.\ with slope vector
$\exp(-\Re(\gamma))$ which is semi-conjugate to $T$. These two affine i.e.m.\ have distinct wandering intervals and are, therefore, not conjugate to each other.

\section{Main consequences of the unique representation property}\label{sec:Extreme points}

In this section we state the main technical lemmas implied by the u.r.p.\ that we will use to prove our main results.
Consider $T$ a self-similar i.e.m.\ satisfying the hypotheses of Theorem \ref{thm:H_minimal}. That is, $M$ has an eigenvalue $\beta$ with $|\beta|>1$ such that $\beta/|\beta|$ is not a root of unity, and that there exists an eigenvector $\Gamma$ for $\beta$ such that $T$ has the u.r.p.\ for $\beta$ and $\Gamma$. This eigenvector will be fixed for the rest of the article and all concepts defined in Section \ref{subsec: fractals associated with sigma and beta} will be associated with it. 

Our next lemma is a slightly more general version of Lemma 7.7 in \cite{CGM}. We omit the proof because it is essentially the same.

\begin{lem}\label{lem:central}
	Let $a \in \A$ and $\tau\in \SS^1$ such that $v_a(\tau)=v_{a,(p,c,s)}(\tau)=
	v_{a,(\bar p,\bar c,\bar s)}(\tau)$, with $(p,c,s)$ and $(\bar p,\bar c,\bar s)$ different elements of $\bA_a$. Then there exist finite constants 
	$0<D_1\leq D_2$ such that if $(\tau_k)_{k\geq 1}$ is a sequence in $\SS^1$ converging to $\tau$ when $k\to\infty$, then
	\begin{align*}
	D_1 &\leq \liminf_{k\to\infty} \frac{|v_{a,(p,c,s)}(\tau_k) - v_{a,(\bar p,\bar c,\bar s)}(\tau_k)|}{\llbracket \tau - \tau_k \rrbracket} \\
	&\leq \limsup_{k\to\infty} \frac{|v_{a,(p,c,s)}(\tau_k) - v_{a,(\bar p,\bar c,\bar s)}(\tau_k)|}{\llbracket \tau - \tau_k \rrbracket} \leq D_2.
	\end{align*}
\end{lem}

The constants $D_1$ and $D_2$ are given by:
\begin{align*}
	D_1 &=\min\{ |z - z'|; z\in E_{a,(p,c,s)}(\tau) , z'\in E_{a,(\bar p,\bar c,\bar s)}(\tau) \}, \\
	D_2 &=\max\{ |z - z'|; z\in E_{a,(p,c,s)}(\tau) , z'\in E_{a,(\bar p,\bar c,\bar s)}(\tau) \}.
\end{align*}

The u.r.p.\ implies that $E_{a,(p,c,s)}$ and $E_{a,(\bar p,\bar c,\bar s)}$ are disjoint. Since both sets are compact, we obtain that $0< D_1 \leq D_2<\infty$.

As a consequence of Lemma~\ref{lem:central} we get:

\begin{lem}\label{lem:psi finito}
The set $\Psi_a$ is finite for each $a\in\A$.
\end{lem}

\begin{proof}
Suppose by contradiction that for some $a\in\A$ the set $\Psi_a$ is infinite and let $(\tau_k)_{k\geq 1}$ be a sequence in $\Psi_a$ that converges to $\tau \in \SS^1$ such that $\tau_k \neq \tau$ for every $k \geq 1$. 

Without loss of generality we may assume that there exist $(p,c,s) \neq (\bar p,\bar c,\bar s)$ in $\bA_a$ such that $E_{a,(p,c,s)}(\tau_k) \neq \varnothing$ and $E_{a,(\bar p,\bar c,\bar s)}(\tau_k) \neq \varnothing $ for every $k \geq 1$. By continuity of $v_a$, $v_{a,(p,c,s)}$ and $v_{a,(\bar p,\bar c,\bar s)}$, we have that $\tau \in \Psi_a$. Indeed, this set is closed.

Consider sequences $z_k\in\F_{a,(p,c,s)}$ and $z'_k\in\F_{a,(\bar p,\bar c,\bar s)}$ attaining the minimum for the direction $\tau_k$, that is, $v_a(\tau_k) = \Re(\tau_k z_k) = \Re(\tau_k z'_k)$. We may assume that $z_k\rightarrow z$ and $z'_k\rightarrow z'$ when $k \to \infty$. Then clearly $z \in E_{a,(p,c,s)}(\tau)$ and $z' \in E_{a,(\bar p,\bar c,\bar s)}(\tau)$. Thus the hypotheses of Lemma~\ref{lem:central} are satisfied.
Since $v_{a,(\bar p,\bar c,\bar s)}(\tau_k) = v_{a,(p,c,s)}(\tau_k)$ for all $k \geq 1$, this lemma ensures that 
$\min\{ |z - z' |; z\in E_{a,( p, c, s)}(\tau) , z'\in E_{a,(\bar p,\bar c,\bar s)}(\tau) \} =0$. Therefore, the intersection
$E_{a,(p,c,s)}(\tau) \cap E_{a,(\bar p,\bar c,\bar s)}(\tau) $ is nonempty, since these sets 
are closed, contradicting the u.r.p.
\end{proof}

\section{\texorpdfstring{The skew product $H$}{The map H}}\label{sec:the map H}

In this section we thoroughly study the map $H$. We always assume the u.r.p.\ for $\beta$ and $\Gamma$. We will see that $H$ is conjugate to a piecewise translation on the interval, that it has finitely many minimal components and that, when restricted to each minimal component, it is an interval exchange map.

Recall that $\bA_a$ is the set of possible triples $(p,c,s) $ in $\A^* \times \A \times \A^*$ such that $\sigma(a)=pcs$ and that $\bA = \bigcup_{a\in\A} \bA_a$.

\begin{lem}\label{lem:itm}
For each $a \in \A$ and $(p, c, s) \in \bar{\A}_a$ the set 
$$J_{a,(p,c,s)}=\{ \tau \in \SS^1; v_a(\tau) = v_{a,(p,c,s)}(\tau) \}$$ 
is a finite union of closed intervals.
Moreover, if $(p,c,s)$ and $(\bar p,\bar c,\bar s)$ are different elements in $\bA_a$, then the interiors of $J_{a,(p,c,s)}$ and $J_{a,(\bar p,\bar c,\bar s)}$ are disjoint.

\end{lem}
\begin{proof}
We prove that each $J_{a, (p, c, s)}$ is closed and has a finite number of connected components. Let $(\tau_k)_{k \geq 1}$ be a sequence in $J_{a, (p, c, s)}$ that converges to $\tau \in \SS^1$. Let $(z_k)_{k \geq 1}$ be a sequence in $\F_{a,(p, c, s)}$ such that $\Re( \tau_k z_k ) = v_a(\tau_k)$. Since $\F_{a,(p, c, s)}$ is compact, we can assume that $(z_k)_{k \geq 1}$ converges to $z \in \F_{a,(p, c, s)}$. Therefore, by continuity of $v_a$ we have that $\Re(\tau z) = v_a(\tau)$. By definition, we conclude that $\tau \in J_{a, (p, c, s)}$ and thus $J_{a, (p, c, s)}$ is closed. 
	
Now, since each of the sets $J_{a, (p, c, s)}$ is closed, then the boundary of a connected component in $J_{a, (p, c, s)}$ is contained in $\Psi_a$. By the u.r.p., $\Psi_a$ is finite (see Lemma \ref{lem:psi finito}). Then each $J_{a, (p, c, s)}$ has finitely many connected components and thus is a finite union of closed intervals.
	
Finally, $J_{a,(p,c,s)} \cap J_{a,(\bar p,\bar c,\bar s)}$ is contained in
$\Psi_a $. So again by finiteness of $\Psi_a$ their interiors are disjoint.
\end{proof}
Notice that, for each $a \in \A$, the union of the closed intervals composing the sets
$J_{a, (p, c, s)}$ covers $\SS^1$. Therefore, if each of such closed intervals is 
redefined to be left-closed and right-open we get a partition of $\SS^1$ by intervals.

Recall that $X = \SS^1 \times \bA$ and that $H\colon X \to X$ is given by
$$H(\tau, (q, a, r)) = (\beta_0^{-1} \tau, (p, c, s))$$
if $v_a(\tau)=v_{a,(p,c,s)}(\tau)$. Equivalently, $H(\tau, (q, a, r)) = (\beta_0^{-1} \tau, (p, c, s))$ if $\tau \in J_{a,(p,c,s)}$. The definition is ambiguous if $\tau \in \Psi_a$ (there is more than one choice for $(p,c,s)$).
Nevertheless, by Lemma \ref{lem:psi finito}, $\Psi$ is finite when the u.r.p.\ holds, so $H$ is well defined and continuous except at finitely many points. 
Therefore, we can fix the ambiguity by setting $H$ to be right-continuous. This is possible since the ambiguities are determined by the boundaries of the closed intervals defining each $J_{a, (p, c, s)}$.
Observe that the definition of $H$ is independent of $q$ and $r$.

\subsection{\texorpdfstring{Orbits of $\bm{H}$ and extreme points}{Orbits of H and extreme points}}
Let $\H$ be the set of all possible maps defined as $H$ by omitting the right-continuous hypothesis, including $H$. As discussed in previous paragraph, this set is clearly finite. 
Also, any element in $\H$ is aperiodic because $\beta_0$ is not a root of unity.
First we notice that:
\begin{rem}\label{rem:exist H}
The existence of $\tilde{H} \in\H$ satisfying 
$\tilde{H}(\tau, (q,a,r)) = (\beta_0^{-1}\tau, (p,c,s))$ is equivalent to the fact that $v_a(\tau)=v_{a,(p,c,s)}(\tau)$, \emph{i.e.}, it is equivalent to the existence of $x\in\S_{a,(p,c,s)}$
such that $v_a(\tau)=\Re(\tau \z_a(x))$. 	
\end{rem}

Let $\pi_{\SS^1} \colon X \to \SS^1$ and $\pi_{\bar{\A}} \colon X \to \bar{\A}$ be the projections to the first and second coordinates of $X$ respectively.

Lemma \ref{lem:intutitionH} below establishes a relation between the orbits of maps in $\H$ with the representations in $\S_a$ of extreme points in $\F_a$ for any direction. Indeed, it tell us that such representations are the same as forward orbits of maps in $\H$. We remark that such orbits are by definition in $\S_a$.

To prove it we recall Lemma 5.6 in \cite{CGM}:

\begin{lem}\label{lem:cont-prop}
If $x=(p^x_m, c^x_m, s^x_m)_{m\geq 1} \in \S_a$ is a representation of an extreme point in $\F_a$ for the direction $\tau$, then the shift 
$S(x)=(p^x_{m+1}, c^x_{m+1}, s^x_{m+1})_{m\geq 1} \in \S_{c_1^x}$ is a representation of an extreme point in $\F_{c_1^x}$ for the direction $\beta_0^{-1}\tau $. 
\end{lem}
This is called the \emph{continuation property}.

\begin{lem}\label{lem:intutitionH}
Let $\tau \in \SS^1$ and $a \in \A$. A sequence $x \in \S_a$ is the representation of an extreme point in $E_a(\tau)$ if and only if there exists $\tilde{H} \in \H$ such that
$$(p_m^x, c_m^x, s_m^x) = \pi_{\bar \A}(\tilde{H}^{m}(\tau, (q,a,r))) \text{ for all } m\geq 1,$$
for any $(q,a,r) \in \bA$. 
\end{lem}
\begin{proof}
Let $x=(p^x_m, c^x_m, s^x_m)_{m\geq 1} \in \S_a$ be the representation of an extreme point for the direction $\tau$. That is, $v_a(\tau) = \Re( \tau\z_a(x))$.
Fix some $m\geq 0$ and put $p^x_0=q$, $c^x_0=a$ and $s^x_0=r$, where $q$ and $r$ are chosen such that $(q,a,r) \in \bA$. 
By the continuation property (Lemma \ref{lem:cont-prop}), the shifted sequence 
$S^{m}(x)$ belongs to $\S_{c^x_m, (p^x_{m+1}, c^x_{m+1}, s^x_{m+1})}$ and is a representation of an extreme point for the direction $\beta_0^{-m}\tau$.
From Remark \ref{rem:exist H}, there exists ${H}_m\in\H$ such that 
$$ {H}_m(\beta_0^{-m}\tau, (p^x_m, c^x_m, s^x_m)) = (\beta_0^{-(m+1)}\tau, (p^x_{m+1}, c^x_{m+1}, s^x_{m+1})).$$
Since the ambiguity points to define $H$ are finite, for some $m_0\geq 1$ the sequence 
$(\beta_0^{-m}\tau, (p^x_{m}, c^x_{m}, s^x_{m}))_{m\geq m_0}$ does not contain any such point.
Thus, $H_m$ can be taken to be $H$ for all $m\geq m_0$. Since $\beta_0$ is not a root of unity, the finite sequence 
$(\beta_0^{-m}\tau, (p^x_{m}, c^x_{m}, s^x_{m}))_{0\leq m \leq m_0 - 1}$ cannot repeat any ambiguity point. 
Thus the map $\tilde{H}$ in $\H$ that is equal to $H_m$ in 
$(\beta_0^{-m}\tau, (p^x_{m}, c^x_{m}, s^x_{m}))$ for $0\leq m \leq m_0 - 1$ 
and equal to $H$ elsewhere is well defined. We conclude that
$$\tilde{H}(\beta_0^{-m}\tau, (p^x_m, c^x_m, s^x_m)) = (\beta_0^{-(m+1)}\tau, (p^x_{m+1}, c^x_{m+1}, s^x_{m+1})), \text{ for all } m\geq 0.$$

Conversely, suppose that $x=(p^x_m, c^x_m, s^x_m)_{m\geq 1} \in \S_a$ 
is obtained from the trajectory by some 
$\tilde{H}\in\H$ of $(\tau, (q,a,r))$ with $(q,a,r) \in \bA$. Set $p^x_0=q$, $c^x_0=a$ and $s^x_0=r$. 

From Remark \ref{rem:exist H}, if $\tilde{H}(\beta_0^{-m}\tau, (p^x_m, c^x_m, s^x_m))= (\beta_0^{-(m+1)}\tau, (p^x_{m+1}, c^x_{m+1}, s^x_{m+1}))$ for $m\geq 0$, then there exists a representation $x_m \in \S_{c^x_m,(p^x_{m+1}, c^x_{m+1}, s^x_{m+1})}$ of an extreme point of the fractal $\F_{c^x_m}$ in the direction
$\beta_0^{-m} \tau$: $v_{c^x_m}(\beta_0^{-m}\tau)=\Re(\beta_0^{-m}\tau \z_{c^x_m}(x_m))$.

Then, using recursively the properties of each $x_m$ and the continuation property we get:
\begin{align*}
v_a(\tau) &= \Re(\tau \z_a(x_0)) \\
&= \Re(\beta^{-1}\tau\Gamma(p^x_1)) + |\beta|^{-1} v_{c^x_1}(\beta_0^{-1}\tau) \\
&= \Re(\beta^{-1}\tau\Gamma(p^x_1)) + |\beta|^{-1} \Re(\beta_0^{-1}\tau z_{c^x_1}(x_1)) \\						& = \Re(\beta^{-1}\tau\Gamma(p^x_1)) + \Re(\beta^{-2}\tau\Gamma(p^x_2)) + |\beta|^{-2} v_{c^x_2}(\beta_0^{-2}\tau) \\
			&= \sum_{m=1}^n \Re( \beta^{-m}\tau\Gamma(p^x_m) ) + |\beta|^{-n} v_{c^x_n}(\beta_0^{-n}\tau)\\
			&=\Re\left(\tau\sum_{m=1}^n \beta^{-m}\Gamma(p^x_m)\right)+ |\beta|^{-n} v_{c^x_n}(\beta_0^{-n}\tau)\\
			&= \Re(\tau \z_a^{(n)}(x))+ |\beta|^{-n} v_{c^x_n}(\beta_0^{-n}\tau).
\end{align*}
Taking the limit when $n \to \infty$ we conclude that $\Re(\tau \z_a(x)) = v_{a}(\tau)$.
\end{proof}

\subsection{Minimal components of maps in \texorpdfstring{$\H$}{H}}
Any map in $\H$ can be visualized as an \emph{interval translation map} (i.t.m.)
We recall that a map defined in an interval $I$ is an i.t.m.\ if it is a piecewise translation with finitely many discontinuities. Contrary to i.e.m.'s these maps need not be injective or surjective. 
Basic properties of i.t.m.\ can be found in \cite{itm}. 

To simplify notations we only illustrate this construction with the right-continuous map $H$. 
For other maps in $\H$ it is analogous. 
Partition the interval $I=[0,1)$ into $|\bA|$ consecutive intervals of the same length, each one associated with an element $(p,c,s) \in \bA$. Call $I_{(p,c,s)}$ such interval and assume it is 
left-closed and right-open. Observe that the substitution $\sigma$ associated with the i.e.m.\ $T$ is injective (indeed, the associated matrices are invertible), so for each 
$(p,c,s) \in \bA$ there exists a unique $a \in \A$ such that $\sigma(a)=pcs$.

The map $H$ naturally induces a map on $I$ that we also call $H$ in the following way. For each $(p,c,s) \in \bA$ let 
${\mathcal i}_{(p,c,s)}\colon \SS^1 \to I_{(p,c,s)}$ be an orientation preserving linear identification of both sets such that ${\mathcal i}_{(p,c,s)}(1)$ is the left extreme point of the interval $I_{(p,c,s)}$. 
If $H(\tau,(q,a,r))=(\beta_0^{-1}\tau,(p,c,s))$ on $X$ then $H({\mathcal i}_{(q,a,r)}(\tau))={\mathcal i}_{(p,c,s)}(\beta_0^{-1}\tau)$ on $I$ (see Figure \ref{fig:explainH}). 
Since $H$ on $X$ has finitely many discontinuities, the map $H$ seen on $I$ is an i.t.m. Moreover, since $\beta_0$ is not a root of unity, $H$ is an aperiodic i.t.m.

\begin{figure}
\centering
\includegraphics[width=0.8\textwidth]{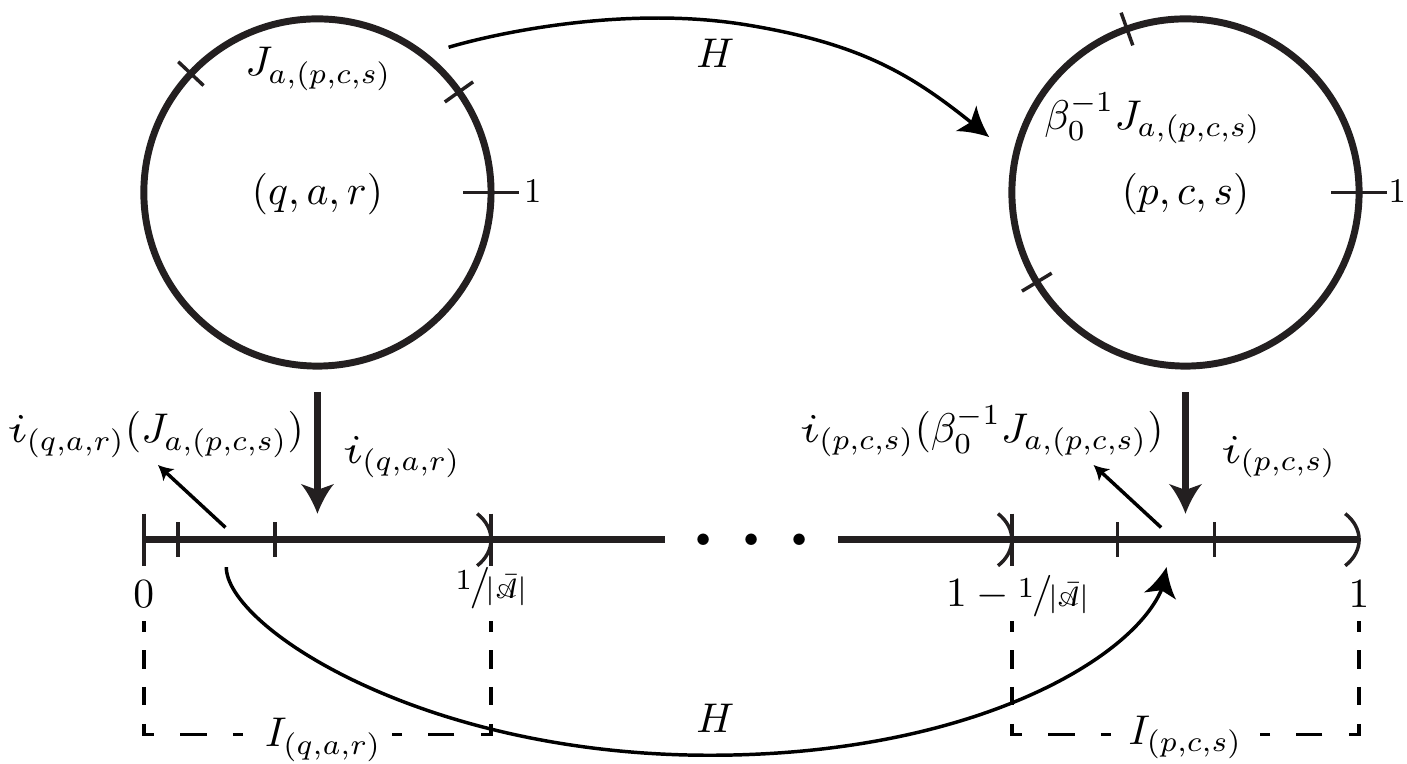}
\caption{The map induced by $H$ on $I$. In the top part we illustrate the action of $H$ on $X=\SS^1 \times \bA$, in the bottom part, the action on $I$.}
\label{fig:explainH}
\end{figure}

We need to define the notion of minimal component for $H$ seen in $I$. Since $H$ is not continuous we need to adapt the classical definition from topological dynamics. This is done by using a standard procedure in i.e.m.\ theory that can be adapted to the context of an i.t.m.\ and which we sketch here. It follows the discussion in Section 2 of \cite{itm}.

First we call $\tilde H \in \H$ the left continuous version of $H$, that is, 
$\tilde H(t)= \lim_{s\nearrow t} H(s)$. Now we define a new space $\hat I$. Let $D$ be the set of discontinuities of $H$ together with its preimages and images by $\tilde H$. 
Then we build the ordered set 
$\hat I=I \cup D^- \cup \{1\}$, where $D^-=\{t^-;t\in D\}$ is a disjoint copy of $D$ putting every point $t^-$ immediately to the left of $t$. That is, we introduce little holes in $I$ at positions $t\in D$ calling the left side of the hole $t^-$ and the right side $t$. The order of $I$ naturally extends to this new set. The set $\hat I$ endowed with the order topology is a compact metric space. Finally, let $\hat H \colon \hat I \to \hat I$ be the map defined as $H$ in $I$, $\hat H(t^-)=(\tilde H(t))^-$ for $t \in D$ and $\hat H(1)$ is defined by continuity (notice that $\hat H$ is increasing in a neighbourhood of $1$).

One proves that $\hat H$ is a continuous map on the compact metric space $\hat I$. Moreover, $\hat H$ leaves $I$ invariant and $\hat H|_{I}$ (the restriction of $\hat H$ to $I$) coincides with $H$ as a map. 

\begin{defn}[\cite{st}]\label{def:minimalset}
We say that $J\subseteq I$ is a \emph{minimal component} for the i.t.m.\ $H$ if $J=\hat J \cap I$, where 
$\hat J$ is a minimal component of $\hat H$. That is, $\hat H(\hat J)=\hat J$ and every point 
in $\hat J$ has a dense orbit by $\hat H$ (for the corresponding topology).
\end{defn}
By definition of $\hat H$, if $\hat J$ is a minimal component for $\hat H$ then its restrictions to $I$ and $D^{-}$ are invariant by $\hat H$. Moreover, if $1 \in \hat J$ then $\hat H(1)=1$. Then, $J=\hat J\cap I$ is strongly invariant by $H$, that is, $H(J)=J$. We also have that 
$J \subseteq \overline{\Orb_H(t)}$ for any $t \in J$, and that two different minimal components are disjoint.

Using the fact that $H$ is an aperiodic i.t.m.\ and Theorem 2.4 in \cite{st} we get:

\begin{lem}\label{lem:finitecomponents}
$H$ has a finite number of minimal components.
\end{lem}

Recall that we have mapped each $t \in I$ to a unique point $(\tau,(p,c,s)) \in X$. This map can be extended to $\hat I$ by sending each $t^-$ to the same $(\tau,(p,c,s)) \in X$ as $t$ for every $t \in D$. We call this map $\mathcal{e}\colon \hat I \to X$.

\begin{lem}\label{lem:finiteunion}
The minimal components of $\hat H$ and $H$ are finite unions of intervals of positive length.
\end{lem}

\begin{proof}
Let $J$ be a minimal component of $H$. Then, by definition $J=\hat J \cap I$ for a minimal component $\hat J$ of $\hat H$. 

For each $(p,c,s) \in \bA$ define 
$\hat J_{(p,c,s)}$ as the projection on $\SS^1$ of 
$\mathcal{e}(\hat J \cap \hat I_{(p,c,s)})$, where $\hat I_{(p,c,s)} = I_{(p,c,s)} \cup \{t^-; t \in D \cap I_{(p,c,s)}  \}$. We will prove that $\bigcup_{(p,c,s) \in \bar{\A}}\hat J_{(p,c,s)}=\SS^1$.

Let $\hat t \in \hat J$ and let $(\tau, (p,c,s)) = \mathcal{e}(\hat t) \in X$.
For any integer $m \geq 0$ we have that the first coordinate of $\mathcal{e}(\hat{H}^m(\hat t))$  is $\beta_0^{-m} \tau$. Since the rotation by $\beta_0^{-1}$ is irrational, 
we get that $\SS^1=\overline{\{ \beta_0^{-m} \}_{m \geq 0}} $. Moreover, we can find subsequences converging to every point in $\SS^1$ from above and below. 
Since $\hat J$ is compact, we obtain that $\SS^1 = \bigcup_{(p,c,s) \in \bar{\A}} \hat J_{(p,c,s)}$. Notice that we have used the convergence in the topology of $\hat I$. 
A similar argument shows that each $\hat J_{(p,c,s)}$ is closed. 
Then, there exists $(p,c,s) \in \bar{\A}$ such that $\hat J_{(p,c,s)}$ contains an open interval 
$K \subseteq \SS^1$. 

Let $\hat K \subseteq \hat J$ be the set of the $\hat t \in \hat J \cap \hat I_{(p,c,s)}$ such that the first coordinate of $\mathcal{e}(\hat t)$ belongs to $K$. Since $K$ is an open interval, there exist $s, r \in I$ such that for every $s < t < r$ either $t$ or $t^-$ belongs to $\hat K$. Since $\hat J$ is closed, we deduce that both $t$ and $t^-$ belong to $\hat K$ for every $s < t < r$, so $\hat K$ is an open interval in $\hat J$.

By minimality, there exists $n \geq 1$ such that $\bigcup_{m = 0}^n \hat H^m(\hat K) = \hat J$. Since $I$ and $D^-$ are invariant for $\hat{H}$, we obtain that $J=\bigcup_{m = 0}^n H^m(\hat K \cap I)$. 
To conclude, we have that $\bigcup_{m = 0}^n \hat H^m(\hat K)$ and $\bigcup_{m = 0}^n H^m(\hat K \cap I)$ are finite unions of intervals of positive length, since the image of an interval by either $\hat H$ or $H$ is a finite union of intervals (recall that it is an i.t.m.)
\end{proof}

Given $(\tau,(p,c,s)) \in X$ there always exists a minimal component contained in the closure of its orbit by $H$. Indeed, let $t \in I$ such that $\mathcal{e}(t) = (\tau,(p,c,s))$. The closure of its orbit by $\hat H$ contains a minimal component 
$\hat J$. Then, $J=\hat J\cap I$ satisfies 
$J \subseteq \overline{\Orb_{\hat H}(t)}\cap I \subseteq \overline{\Orb_H(t)}$
(notice that the topology of $\hat I$ is stronger than the one of $I$).
We conclude by mapping back these objects to $X$. 	

\begin{lem}\label{lem:minimalarrival}
There exists $N\geq 0$ such that for any 
$(\tau,(p,c,s)) \in X$ we have that $H^{N}(\tau,(p,c,s))$ belongs to a minimal component contained in the closure of its orbit by $H$. In particular, there exists a unique minimal component contained in the closed orbit of any point in $X$. 
\end{lem}
\begin{proof}
We prove the lemma using the map $H$ as seen on $I$. Consider the element $t \in I$ such that $\mathcal{e}(t) = (\tau,(p,c,s))$.
Let $J$ be a minimal component contained in $\overline{\Orb_H(t)}$. 
By definition $J=\hat J\cap I$, where $\hat J$ is a minimal component for $\hat H$ with 
$\hat J \subseteq \overline{\Orb_{\hat H}(t)}$ as discussed just before the lemma. 

By Lemma \ref{lem:finiteunion}, we have that $\hat J$ has nonempty interior. Since $\hat J$ is contained in $\overline{\Orb_{\hat H}(t)}$, there exists $n\geq 0$ such that $\hat H^n(t)$ attains the interior of $\hat J$. This implies in particular that $\overline{\Orb_{\hat H}(t)}$ contains a unique minimal component. Thus, $\overline{\Orb_{\hat H}(t)}=\bigcup_{m\geq 0} \hat H^{-m}(\hat J)$. But $\hat J$ has nonempty interior and $\hat H$ is continuous, so by compactness there exists $N\geq 1$ such that 
$\overline{\Orb_{\hat H}(t)}=\bigcup_{m=0}^N \hat H^{-m}(\hat J)$ and then $\hat H^N(t) \in \hat J$. Using a similar argument as the one developed in the proof of the previous lemma to pass from $\hat J$ to $J$, one deduces that $H^N(t) \in J$. 

To conclude that $N$ can be chosen uniformly, we observe from Lemma \ref{lem:finitecomponents} that $H$ and $\hat H$ have finitely many minimal components. 
\end{proof}

\subsection{Proof of Theorem~\ref{thm:H_minimal}}

Let $\Lambda_H$ be the limit set of $H$:
$\Lambda_H = \bigcap_{m \geq 1} H^m(X)$. As a consequence of the previous lemma and the fact that minimal components for $H$ are strongly invariant (see the comment after the definition of minimal components), there exists $N\geq0$ such that $H^m(X)=H^N(X)$ if $m \geq N$.
In the nomenclature of \cite{itm}, this property means that $H$ is of \emph{finite type}. 
By Lemma~\ref{lem:minimalarrival}, every point $(\tau,(p,c,s))$ attains the minimal component in the closure of its orbit in $N$ steps. Moreover, $H$ is surjective when restricted to a minimal component.
Therefore, $\Lambda_H=H^{N}(X)$ is equal to the disjoint union of the minimal components of $H$.
We collect all this observations in the following Corollary for future reference.

\begin{cor}\label{thm:H^N_minimal}
There exists $N \geq 0$ such that $\Lambda_H=H^{N}(X)$. Moreover, $\Lambda_H$ is equal to the disjoint union of the minimal components of $H$ and thus is a finite union of intervals of positive length. 
\end{cor}

\begin{proof}[Proof of Theorem~\ref{thm:H_minimal}]
It was already proved in Lemmas~\ref{lem:finitecomponents} and \ref{lem:finiteunion} that $H$ has finitely many minimal components each of which is a union of intervals. We need to prove that the restriction of $H$ to a minimal component is a minimal i.e.m.
At each minimal component $H$ is surjective, but since it is an i.t.m.\ 
it is also injective, so it is an i.e.m.
\end{proof}

\section{Construction of minimal points}\label{sec:proofs}
We continue under the same assumptions of the previous sections. In particular, we fix the eigenvector $\Gamma$ used to define fractals and related concepts in Section \ref{subsec: fractals associated with sigma and beta}.

The following result is implied by Lemma 7.4 in \cite{CGM}, since Definition~\ref{def:good} is weaker than the one in that article.
\begin{lem}\label{lem:almost-good-eigenvector}
Almost every $\tau \in \SS^1$ is a good direction.
\end{lem}
A direct consequence of the lemma is that 
almost every eigenvector $\gamma$ in the complex space generated by $\Gamma$ is good.

\subsection{Technical lemmas}\label{sec:technicallemmas}
Our first lemma is a slight modification of Lemma 5.7 in \cite{CGM}. We omit the proof since it is almost identical.

\begin{lem}\label{lem:converg-exponencial}
There exists a constant $C > 0$ such that 
for all $\tau\in\SS^1$, $a\in \A$, $(p,c,s)\in\bA_a$ and $n\geq 0$, 
$$0\leq v^{(n)}_{a,(p,c,s)}(\tau) - v_{a,(p,c,s)}(\tau) \leq C |\beta|^{-n}.$$
\end{lem}

We will also need a stronger result. 

\begin{lem}\label{converg-exponencial-2}
There exists a constant $C > 0$ such that for all $\tau \in \SS^1$, $a\in \A$, $(q,a,r) \in \bA$ and $n\geq 1$, we have that if $H^m(\tau,(q,a,r))= (\beta_0^{-m}\tau, (p_{m}, c_{m},s_{m}))$ for 
$m \geq 1$ and $x=(p_{m},c_{m},s_{m})_{m\geq 1} \in \S_{a}$, then
$$0 \leq \Re(\tau \z_a^{(n)}(x))-v_a(\tau) \leq C|\beta|^{-n}.$$
Moreover, if $x\in \S_{a,(p,c,s)}$ then
$$0 \leq \Re(\tau \z_a^{(n)}(x))-v_{a,(p,c,s)}(\tau) \leq C|\beta|^{-n}.$$

\end{lem}

\begin{proof}
By Lemma \ref{lem:intutitionH}, $v_a(\tau) = \Re( \tau \z_a(x))$. Then, 
\begin{align*}
v_a(\tau) &= \Re\left(\tau\sum_{m\geq 1} \beta^{-m}\Gamma(p_m)\right) \\ 
&= \Re\left(\tau\sum_{m=1}^n \beta^{-m} \Gamma(p_m)\right) + |\beta|^{-n}\Re\left( \beta_0^{-n}\tau \sum_{m\geq n+1} \beta^{n-m}\Gamma(p_{m}) \right) \\
&= \Re(\tau \z_a^{(n)}(x)) + |\beta|^{-n}\Re\left( \beta_0^{-n}\tau \sum_{m\geq n+1} \beta^{n-m}\Gamma(p_{m}) \right).
\end{align*}
We conclude using that the series 
$\sum_{m\geq n+1} \beta^{n-m}\Gamma(p_{m})$ is uniformly bounded with respect to
$a\in\A$, $(q,a,r)\in\bA$ and $n\geq 1$. The second statement of the lemma is analogous. 
\end{proof}

The next lemma is the crucial step in the proofs of our main results. 
It will allow to characterize 
the prefix-suffix decompositions of minimal sequences. 

\begin{lem}\label{lem:technical}
Let $\tau \in \SS^1$ be a good direction and $a \in \A$. Assume $(m_k)_{k\geq 1}$ is an increasing sequence of positive integers such that 
$v_{a}(\beta_0^{m_k}\tau)=v_{a,(p,c,s)}(\beta_0^{m_k}\tau)$ for all $k \geq 1$ with $(p,c,s) \in \bA_a$. There exists $k_0 \geq 1$ such that if, for some $k \geq k_0$,
$x\in\S_a$ satisfies
$v_a^{(m_k)}(\beta_0^{m_k}\tau)= \Re(\beta_0^{m_k}\tau\z_a^{(m_k)}(x))$, then $x \in \S_{a, (p,c,s)}$. 
\end{lem}

\begin{proof}
We proceed by contradiction. Without loss of generality, suppose that 
there exists $(\bar p,\bar c,\bar s)\neq (p,c,s)$ in $\bA_a$ and a sequence $x_{k} \in \S_{a,(\bar p,\bar c,\bar s)}$ such that $v_a^{(m_k)}(\beta_0^{m_k}\tau)= \Re(\beta_0^{m_k}\tau\z_a^{(m_k)}(x_k))$ for all $k \geq 1$.

Consider also the sequence $(y_{k})_{k\geq 1}$ given by 
$$y_{k}=( \pi_{\bA} ( H^m(\beta_0^{m_k}\tau,(q,a,r)) ))_{m\geq 1} \text{ for a fixed } (q,a,r)\in \bA.$$

By Lemma \ref{lem:psi finito} the set $\Psi_a$ is finite. Then, since $\beta_0$ is not a root of unity, after extracting a subsequence 
we can assume that $\beta_0^{m_k} \tau\notin\Psi_a$. Then, by hypothesis and definition of $H$, it follows that each $y_k \in \S_{a,(p,c,s)}$ and
$$v_{a,(\bar p,\bar c,\bar s)}(\beta_0^{m_k}\tau) > v_{a,(p,c,s)}(\beta_0^{m_k}\tau) \text{ for all } k \geq 1.$$

Without loss of generality we assume that $\beta_0^{m_k}\tau \to \xi \in \SS^1$. 
By continuity of $v_a$ and $v_{a,(p,c,s)}$, $v_a(\xi)=v_{a,(p,c,s)}(\xi)$. On the other hand, 
by Lemma \ref{lem:converg-exponencial} used in the sequence $(x_{k})_{k\geq 1}$
we have that 
$v_a(\xi)=v_{a,(\bar p,\bar c,\bar s)}(\xi)$ and thus $\xi\in \Psi_a$. 

Therefore, the hypotheses of Lemma \ref{lem:central} hold and there exists 
$k_0\geq1$ such that for every $k \geq k_0$:
$$v_{a,(\bar p,\bar c,\bar s)}(\beta_0^{m_k}\tau) - v_{a,(p,c,s)}(\beta_0^{m_k}\tau) \geq \frac{1}{2} D \llbracket \xi - \beta_0^{m_k}\tau \rrbracket,$$
where $ D $ is strictly positive.

Since, by hypothesis, $\Re(\beta_0^{m_k}\tau \z_a^{(m_k)}(x_{k})) = v_{a,(\bar p,\bar c,\bar s)}^{(m_k)}(\beta_0^{m_k}\tau) \geq v_{a,(\bar p,\bar c,\bar s)}(\beta_0^{m_k}\tau)$, we get:
\begin{equation}\label{eq:lemtechnicalx}\Re(\beta_0^{m_k}\tau \z_a^{(m_k)}(x_{k})) - v_{a,(p,c,s)}(\beta_0^{m_k}\tau) \geq \frac{1}{2} D \llbracket \xi - \beta_0^{m_k}\tau \rrbracket.\end{equation}

On the other hand, by Lemma \ref{converg-exponencial-2}, we obtain that:
\begin{equation}\label{eq:lemtechnicaly}-\Re(\beta_0^{m_k}\tau \z_a^{(m_k)}(y_{k}))+ v_{a,(p,c,s)}(\beta_0^{m_k}\tau)\geq - C|\beta|^{-m_k}.\end{equation}

Taking \eqref{eq:lemtechnicalx}${}+{}$\eqref{eq:lemtechnicaly} yields:
\begin{align*}
& \Re(\beta_0^{m_k}\tau \z_a^{(m_k)}(x_{k})) - \Re(\beta_0^{m_k}\tau \z_a^{(m_k)}(y_{k})) \\
&\geq \frac{1}{2} D \llbracket \xi - \beta_0^{m_k}\tau \rrbracket - C|\beta|^{-m_k} \\
& \geq |\beta|^{-m_k} \left(\frac{1}{2} D |\beta|^{m_k} \llbracket \xi - \beta_0^{m_k}\tau \rrbracket - C \right).
\end{align*}

Finally, since $\tau$ is a good direction, then $|\beta|^{m_k} \llbracket \xi - \beta_0^{m_k}\tau\rrbracket \to \infty$ as $k\to \infty$.
Therefore, if $k$ is sufficiently large,
$\Re(\beta_0^{m_k}\tau \z_a^{(m_k)}(x_{k})) - \Re(\beta_0^{m_k}\tau \z_a^{(m_k)}(y_{k})) >0$, which contradicts the hypothesis that $v_a^{(m_k)}(\beta_0^{m_k}\tau) = \Re(\beta_0^{m_k}\tau\z_a^{(m_k)}(x_{k}))$ for all $k \geq 1$.
\end{proof}

\begin{defn}\label{def:localminimal}
Let $a \in \A$ and an integer $n \geq 0$. By changing the indices of the letters we can decompose 
the word $\sigma^n(a)$ as a \emph{pointed word} $\sigma^n(a)=w=w_{-N}\ldots w_{-1} \cdot w_0 \ldots w_{N'}$, where $N, N' \geq 0$. We use the notation $w=S^N(\sigma^n(a))$ to refer to this kind of decomposition.

We say that $w=S^{N}(\sigma^n(a))$ is \emph{minimal} for $\sigma^n(a)$ and the vector $\gamma$ if 
$\Re(\gamma_m(w)) \geq 0$ for all $-N \leq m \leq N'$. 
\end{defn}
Observe that this is equivalent to $\Re(\gamma(w_{-N} \ldots w_{-1})) \leq \Re(\gamma(w_{-N} \ldots w_{m}))$ for every $-N \leq m \leq N'$ (see Lemma 4.3 in \cite{CGM}), so $w_{-N} \ldots w_{-1}$ is a proper prefix of $\sigma^n(a)$ satisfying $\Re(\gamma(w_{-N} \ldots w_{-1})) \leq \Re(\gamma(w'))$ for any proper prefix $w'$ of $\sigma^n(a)$.

The next lemma provides a finite prefix-suffix decomposition for $w=S^N(\sigma^n(a))$.
\begin{lem}
Let $a \in \A$, $n \geq 1$ and $w = S^N(\sigma^n(a))$ with $N \leq |\sigma^n(a)| - 1$. That is, $w = w_{-N} \ldots w_{-1} \cdot  w_0 \ldots w_{|\sigma^n(a)| - 1 - N}$. Then, there exists a finite prefix-suffix decomposition $(p_m, c_m, s_m)_{0 \leq m \leq n - 1}$ such that $\sigma(c_{m+1}) = p_m c_m s_m$ for every $0 \leq m \leq n -1$ and $p_{n-1}c_{n-1}s_{n-1} = \sigma(a)$ satisfying $w = \sigma^{n-1}(p_{n-1}) \ldots p_0 \cdot c_0 s_0 \ldots \sigma^{n-1}(s_{n-1})$. In other words, finite words that are shifts of iterates of letters have an analogue of a prefix-suffix decomposition.
\end{lem}

\begin{proof}
	If $n = 1$, then $w = p_0 \cdot c_0 s_0$ with $(p_0, c_0, s_0) \in \A^* \times \A \times \A^*$ satisfying $p_0 c_0 s_0 = \sigma(a)$, so the result holds in this case.
	
	We now assume that $n \geq 2$ and proceed by induction. That is, we assume that any shift of $\sigma^{n-1}(a)$ has a finite prefix-suffix decomposition. Let $w' = S^{N'}(\sigma^{n-1}(a))$, where $N' \leq |\sigma^{n-1}| - 1$ is chosen so $w = S^k(\sigma(w'))$ with the minimum possible $k \geq 0$. Let $(p_m', c_m', s_m')_{0 \leq m \leq n - 2}$ be the finite prefix-suffix decomposition of $w'$. We have that $w' = \sigma^{n-2}(p_{n-2}') \ldots p_0' \cdot c_0' s_0' \ldots \sigma^{n-2}(s_{n-2}')$. By minimality of $k$, we conclude that $k \leq |\sigma(c_0')| - 1$. We write the word $S^k(\sigma(c_0'))$ as $p_0 \cdot c_0  s_0$ and then define the finite prefix-suffix decomposition of $w$ as $(p_0, c_0, s_0)(p_0', c_0', s_0') \ldots (p_{n-2}', c_{n-2}', s_{n-2}')$. By definition of $k$ and $w'$, this sequence satisfies the desired properties.
\end{proof}

The following lemma allows to relate the finite prefix-suffix decompositions of minimal words with extreme points of finite order.

\begin{lem}\label{lem:minimalprefix}
Let $a \in \A$, $n \geq 1$ and $w = S^N(\sigma^n(a))$ with $N \leq |\sigma^n(a)| - 1$.
Assume $w$ is minimal for $\sigma^n(a)$ and $\gamma$ and let $(p_m, c_m, s_m)_{0 \leq m \leq n-1}$ be its finite prefix-suffix decomposition. If $\gamma=\tau\Gamma$ for $\tau\in\SS^1$ and the first $n$ coordinates of $x \in \S_a$ coincide with $(p_{n-m}, c_{n-m}, s_{n-m})_{1 \leq m \leq n}$, then 
$v_a^{(n)}(\beta_0^{n}\tau) = \Re(\beta_0^n\tau \z_a(x))$.
\end{lem}
\begin{proof}
	Put $w=w_{-N}\ldots w_{N'}$ with $N, N' \geq 0$. By definition of the finite prefix-suffix decomposition, we have that $\sigma^{n-1}(p_{n-1}) \ldots p_0 = w_{-N}w_{-N+1}\ldots w_{-1}$. 
	Applying $\gamma$, using \eqref{eq:gamma(sigma^n)} and taking real parts, we obtain:
	$$|\beta|^{n}\Re\left(\sum_{m=1}^{n}\beta^{-m}\gamma(p_{n-m})\right) = |\beta|^n \Re(\beta_0^n\tau \z_a^{(n)}(x)) \leq \Re(\gamma(w'))$$
	for every prefix $w'$ of $\sigma^n(a)$.	
	On the other hand, for any $y \in \S_a$ there exists a prefix $w'$ of $\sigma^n(a)$ such that $|\beta|^{n}\Re(\beta_0^n\tau \z_a^{(n)}(y)) = \Re(\gamma(w'))$. Indeed, we can take $w' = \sigma^{n-1}(p_{n-1}^y)\ldots p_0^y$.
	Thus $\Re(\beta_0^n \tau \z_a^{(n)}(x)) \leq \Re(\beta_0^n \tau \z_a^{(n)}(y))$ for all $y \in \S_a$.
\end{proof}

The following corollary of previous lemma was implicit in the proof of Lemma 5.13 in \cite{CGM}.

\begin{cor}\label{cor:minimalprefix}
Let $(p_m, c_m, s_m)_{m \geq 0}$ be the prefix-suffix decomposition of a minimal sequence $\omega$ for the vector $\gamma = \tau\Gamma$ for some $\tau \in \SS^1$. Let $a = c_n$ and $x \in \S_a$ be such that its first $n$ coordinates coincide with $(p_{n-k}, c_{n-k}, s_{n-k})_{1 \leq k \leq n}$. Then $v_a^{(n)}(\beta_0^n \tau)=\Re(\beta_0^n \tau \z_a^{(n)}(x))$.
\end{cor}

Finally, using the following lemma we only have to consider minimal sequences with infinitely many non-empty prefixes and suffixes. This argument was given in the proof of Proposition 7.8 in \cite{CGM}, but we state it here for convenience.

\begin{lem}\label{lem:infinitenonempty}
Let $(p_m, c_m, s_m)_{m \geq 0}$ be the prefix-suffix decomposition of a minimal sequence $\omega$ for the vector $\gamma = \tau\Gamma$ for some $\tau \in \SS^1$. Then, there exist infinitely many $m \geq 0$ such that $s_m \neq \varepsilon$. Analogously, there exist infinitely many $m \geq 0$ such that $p_m \neq \varepsilon$.
\end{lem}

\begin{proof}
Assume by contradiction that $s_{n_0 + m} = \varepsilon$ for some integer $n_0 \geq 0$ and every $m \geq 0$. We will show that $(p_m, c_m, s_m)_{m \geq 0}$ is eventually periodic, which contradicts Lemma 5.8 in \cite{CGM}. We have that $\sigma(c_{n_0 + m + 1}) = p_{n_0 + m} c_{n_0 + m}$ for every $m \geq 0$. Then, for every $m \geq 0$, the value of $c_{n_0 + m + 1}$ determines a unique possible value for $p_{n_0 + m}$ and $c_{n_0 + m}$. By induction, it is easy to see that $(p_m, c_m, s_m)_{m \geq n_0}$ is periodic. Proving that infinitely many $p_m$'s are nonempty is completely analogous.
\end{proof}

\subsection{Proof of Theorem \ref{thm:enter minimal}}
We have already proved in Theorem \ref{thm:H_minimal} that the restriction of $H$ to each minimal component corresponds to a minimal i.e.m. In this way we can refer to the inverse of $H$ on each minimal component. 

Under the hypothesis of Theorem \ref{thm:enter minimal} we need to prove that: for every good direction $\tau \in \SS^1$ and every minimal sequence $\omega \in \Omega_\sigma$ for $\gamma=\tau \Gamma$,
its prefix-suffix decomposition $(p_m, c_m, s_m)_{m\geq 0}$ satisfies:
\begin{itemize}
\item[(a)] for some $n \geq 0$, $(\beta_0^{n}\tau, (p_{n}, c_n, s_n))$ belongs to a minimal component of $H$;
\item[(b)] $(\beta_0^{m}\tau, (p_m, c_m, s_m)) = H^{-m+n}(\beta_0^{n}\tau, (p_{n}, c_n, s_n))$ for all $m \geq n$.
\end{itemize}

\begin{proof}[Proof of Theorem \ref{thm:enter minimal}]
Let $\tau \in \SS^1$ be a good direction. 
We claim there exists $m_0 \geq 1$ such that for every $m \geq m_0 + 1$ we have
$$H(\beta_0^m\tau,(p_m,c_m,s_m))=(\beta_0^{m-1}\tau,(p_{m-1},c_{m-1},s_{m-1})).$$ If this holds, from Theorem~\ref{thm:H_minimal} we have that 
$H^N$ of any point of $\SS^1\times \bA$ is contained in a minimal component of $H$, where $N\geq 0$ is a universal constant. Therefore, by taking $n = m_0 + N + 1$ we get (a) and (b).

We prove the claim by contradiction. Assume there exists an increasing sequence of integers 
$(m_k)_{k\geq 1}$ such that 
$$H(\beta_0^{m_k}\tau,(p_{m_k},c_{m_k},s_{m_k}))\neq(\beta_0^{{m_k}-1}\tau,(p_{{m_k}-1},c_{{m_k}-1},s_{{m_k}-1})).$$
Without loss of generality we may assume that for all $k\geq 1$:
$c_{m_k}=a$ and there exists $(p,c,s)\neq (\bar p,\bar c,\bar s)$ in $\bA_a$ such that
\begin{itemize}
\item $H(\beta_0^{m_k}\tau, (p_{m_k}, c_{m_k}, s_{m_k})) 
= (\beta_0^{m_k-1}\tau, (p,c,s))$ and
\item $(p_{m_k-1}, c_{m_k-1}, s_{m_k-1}) = (\bar p,\bar c,\bar s)$.
\end{itemize} 
By Remark \ref{rem:exist H} we have that 
$v_a(\beta_0^{m_k}\tau)=v_{a,(p,c,s)}(\beta_0^{m_k}\tau)$ for all $k\geq 1$.

On the other hand, let $x_k \in \S_{a,(p,c,s)}$ be a point such that its first $m_k$ coordinates coincide with $(p_{m_k-m},c_{m_k-m},s_{m_k-m})_{1 \leq m \leq m_k}$. Since $(p_m, c_m, s_m)_{m\geq 0}$ is the prefix-suffix decomposition of a minimal sequence, by Corollary \ref{cor:minimalprefix} we have that
$v_a^{(m_k)}(\beta_0^{m_k}\tau)=\Re(\beta_0^{m_k}\tau\z_a^{(m_k)}(x_k))$ for all $k\geq 1$.
This contradicts Lemma \ref{lem:technical} since $(p,c,s)\neq (\bar p,\bar c,\bar s)$.
\end{proof}

\subsection{Proof of Theorem \ref{thm:minimal points}}

Under the hypotheses of Theorem \ref{thm:minimal points} we need to prove that,
for every good direction $\tau \in \SS^1$, if $(\tau,(p_0,c_0,s_0))$ belongs to a minimal component of $H$ and $H^{-m}(\tau,(p_0,c_0,s_0)) = (\beta_0^m\tau, (p_m,c_m,s_m))$ for every $m\geq 0$, then 
$(p_m,c_m,s_m)_{m\geq 0}$ is the prefix-suffix decomposition of some shift of a minimal sequence for the vector $\gamma=\tau\Gamma$. Let $\T$ be the set of sequences in $(\A^* \times \A \times \A^*)^\N$ that are the prefix-suffix decomposition of a point in $\Omega_\sigma$. It is not difficult to prove that:
$$\T = \{ (p_m,c_m,s_m)_{m \geq 0}; \sigma(c_m) = p_{m-1} c_{m-1} s_{m-1}, m \geq 1 \}.
$$

\begin{proof}[Proof of Theorem \ref{thm:minimal points}]
Let $\tau \in \SS^1$ be good direction. Consider a sequence
$\omega \in \Omega_\sigma$ such that its prefix-suffix decomposition is the one given in the statement of the theorem, \emph{i.e.}, $(p_m, c_m,s_m)_{m\geq 0}$.

We will start by proving that $s_m$ is not empty for infinitely many $m \geq 0$. An analogous proof shows that $p_m$ is not empty for infinitely many $m \geq 0$.

Assume by contradiction that the suffixes are eventually empty. Then, we have that $(p_m, c_m, s_m)_{m \geq 0}$ is eventually periodic as in the proof of Lemma \ref{lem:infinitenonempty}. Thus, there exists $m_0 \geq 0$ and $\ell \geq 0$ such that $(p_{m+k\ell}, c_{m+k\ell}, s_{m+k\ell}) = (p_m, c_m, s_m)$ for every $m \geq m_0$ and $k \geq 0$. 

Let $a = c_{m_0}$ and $x_k \in \S_a$ such that its first $k\ell$ coordinates are
$$(p_{m_0+1+k\ell-m}, c_{m_0+1+k\ell-m}, s_{m_0+1+k\ell-m})_{1 \leq m \leq k\ell}.$$
By Lemma \ref{converg-exponencial-2}, we have that $\Re(\beta_0^{m_0 + k\ell}\tau \z_a^{(k\ell)}(x_k)) - v_a(\beta_0^{m_0 + k\ell} \tau) \leq C |\beta|^{-k\ell}$. Let $x \in \S_a$ be the limit of $(x_k)_{k \geq 1}$, which is periodic. By taking appropriate subsequences we get that $x$ is an extreme point in $\F_a$ for some direction in $\SS^1$. This contradicts Lemma 5.8 in \cite{CGM} which states that eventually periodic elements cannot represent extreme points.
\smallskip

The sequence $(p_m, c_m, s_m)_{m \geq 0}$ induces a partition of the non-zero integers in the following way: the set $A_m$ is defined by the coordinates covered by $\sigma^m(p_m)$ or $\sigma^m(s_m)$ in $\sigma^{m}(p_m) \ldots p_0 \cdot c_0 s_0 \ldots \sigma^{m}(s_{m})$, where the dot separates negative and non-negative coordinates. Since infinitely many $p_m$'s and $s_m$'s are nonempty, we have that $\bigcup_{m \geq 0} A_m = \Z \setminus \{0\}$. We define $u_m(\omega) = \min_{n \in A_m} \Re(\gamma_n(\omega))$, $B_m = \bigcup_{n \leq m} A_n$ and $v_m(\omega) = \min_{n \in B_m} \Re(\gamma_n(\omega))$.

Suppose now by contradiction that $\omega$ is not in the trajectory by the shift of a minimal sequence for the vector $\gamma=\tau\Gamma$. Then, there exists an increasing sequence of integers $(m_k)_{k \geq 1}$ such that $(u_{m_k-1}(\omega))_{k \geq 1}$ is strictly decreasing and equal to $(v_{m_k-1}(\omega))_{k \geq 1}$. Let $n_k \in A_{m_k-1}$ such that $\Re(\gamma_{n_k}(\omega)) = u_{m_k-1}(\omega)$. Without loss of generality, we assume that $c_{m_k} = a$ for every $k \geq 1$.

Let $(p_m^{(k)}, c_m^{(k)}, s_m^{(k)})_{m \geq 0}$ be the prefix-suffix decomposition of $S^{n_k}(\omega)$. We have that $c_{m_k}^{(k)} = c_{m_k} = a$ and that, by definition, 
$v_a^{(m_k)}(\beta_0^{m_k}\tau)=\Re(\beta_0^{m_k}\tau \z_a(x_k))$, where $x_k \in \S_a$ is such that its first $m_k$ coordinates are $(p^{(k)}_{m_k-m},c^{(k)}_{m_k-m},s^{(k)}_{m_k-m})_{1\leq m \leq m_k}$. Since by definition $n_k \in A_{m_k-1}$, we have that:
$$(p_{m_k-1}^{(k)}, c_{m_k-1}^{(k)}, s_{m_k-1}^{(k)}) \neq (p_{m_k-1}, c_{m_k-1}, s_{m_k-1}).$$
We may assume that for every $k \geq 1$:
$({p}_{m_k-1}, {c}_{m_k-1}, {s}_{m_k-1} ) = (p,c,s)$ and
$({p}^{(k)}_{m_k-1}, {c}^{(k)}_{m_k-1}, {s}^{(k)}_{m_k-1} ) = (\bar p,\bar c,\bar s)$.

By definition of $H$, $(p,c,s)$ is chosen so that 
$v_a(\beta_0^{m_k}\tau)=v_{a,(p,c,s)}(\beta_0^{m_k}\tau)$ for all $k \geq 1$.	
This fact contradicts Lemma \ref{lem:technical} since $(p,c,s)\neq (\bar p,\bar c,\bar s)$.

\end{proof}

\subsection{Proof of Corollary \ref{cor:finite min points}} We start by showing that the orbits of minimal sequences are finite:

\begin{lem}\label{cor:finite min orbits}
Given a good direction $\tau\in\SS^1$, there are finitely many orbits of minimal sequences for the eigenvector $\gamma=\tau\Gamma$.	
\end{lem}
\begin{proof}

Let $\Omega_\sigma^{\tau}$ be the set of sequences in $\Omega_\sigma$ whose prefix-suffix decomposition is the projection on $\bA$ of $(H^{-m}(\tau,(q,a,r)))_{m\geq 0}$ for some $(\tau,(q,a,r))$ belonging to a minimal component of $H$. Since $\bA$ is finite, we obtain that $\Omega_\sigma^\tau$ is finite.

Let $(p_m,c_m,s_m)_{m\geq 0}$ be the prefix-suffix decomposition of a minimal sequence $\omega$ for the vector $\gamma=\tau\Gamma$. By Theorem \ref{thm:enter minimal}
there exists $n\geq 0 $ such that
\begin{itemize}
	\item $(\beta_0^{n}\tau, (p_{n}, c_{n}, s_{n}))$ 
	belongs to a minimal component of $H$ and
	\item $(\beta_0^m \tau, (p_m, c_m, s_m)) = H^{-m+n}(\beta_0^n \tau, (p_n, c_n, s_n))$ for every $m \geq n$.
\end{itemize}
Let $\bar \omega \in \Omega_\sigma^\tau$ be the sequence whose prefix-suffix decomposition 
is the projection on $\bA$ of $(H^{-m+n}(\beta_0^{n}\tau, (p_{n}, c_{n}, s_{n})))_{m\geq 0}$. We have that the prefix-suffix decomposition of $\omega$ and $\bar \omega$ coincide for every $m \geq n$. But, by Lemma \ref{lem:infinitenonempty}, infinitely many $s_m$'s and $p_m$'s are nonempty. Therefore $\omega$ belongs to the orbit of $\bar \omega$ by the shift action on $\Omega_\sigma$. This concludes the proof.
\end{proof}

Fix a slope vector $\ell=(\ell_a;a\in\A)\in\R^{\A}$ and for $\omega=(\omega_m)_{m\in\Z}\in\Omega_\sigma$, denote $\ell_n(\omega) = \prod_{m=0}^{n - 1} \ell_{\omega_m}$ and
$\ell_{-n}(\omega) = \prod_{m=-n}^{-1} \ell_{\omega_m}^{-1}$ for $n \geq 0$.
Let
$$\Sigma(\omega, \ell)= \sum_{n \in \Z} \ell_n(\omega),$$
which might be equal to $\infty$ (observe that every term of the series is positive).
If $w=w_{-N} \ldots \omega_{-1} \cdot \omega_0 \ldots \omega_{N'}$ is a finite (pointed) word, we similarly define $\ell_n(w) = \prod_{m = 0}^{n-1} \ell_{w_m}$ and $\ell_{-n}(w) = \prod_{m = -n}^{-1} \ell_{w_m}^{-1}$ for $n \geq 0$. Letting $|w|_+=N'+1$ and
$|w|_- = N$, we set $\Sigma(w, \ell)=\sum_{n=-|w|_-}^{|w|_+} \ell_n(w)$.
\medbreak
The main result of \cite{CGM}, Theorem 7.1, states that if $\omega$ is a minimal sequence for a good eigenvector $\gamma$ and $\ell=\exp(-\Re(\gamma))$, then $\Sigma(\omega, \ell) < \infty$. 
In the next lemma we will prove a similar result for a sequence of finite words which are minimal in the sense of Definition~\ref{def:localminimal} and eigenvectors which are good in the sense of Definition~\ref{def:good}.

\begin{lem}\label{lem:cota_minimos}
	Let $\gamma=\tau \Gamma$ be a good eigenvector and let $\ell=\exp(-\Re(\gamma))$. Fix $a\in\A$ and for every $n \geq 0$ let $w^{(n)}$ be minimal for $\sigma^n(a)$ and $\gamma$. 
	Then, there exists a constant $K>0$ such that 
	$\Sigma(w^{(n)},\ell) < K$ for all $n \geq 0$.

\end{lem}

Let us remark that the proof of this lemma uses the same techniques as those of Theorem
7.1 in [CGRM17], 
but the present result seems stronger. Indeed, observe that if $\omega$
is a minimal sequence for the vector $\gamma$ and $(p_m, c_m, s_m)_{m \geq 0}$
is its prefix-suffix decomposition, then
for all $n \geq 0$ the pointed word
whose prefix-suffix decomposition is $(p_m, c_m, s_m)_{1 \leq m \leq n-1} $
is minimal for $\sigma^n(c_n)$ and $\gamma$.
Therefore, this lemma easily implies Theorem
7.1 of [CGRM17], namely, that $ \Sigma(\omega,\ell) < \infty. $
We tried hard to prove the converse, unsuccessfully. For this reason, a new proof, although with very similar arguments to that of Theorem
7.1 in [CGRM17], seemed to us unavoidable. Moreover, we
need to account for the different definition of good eigenvector.

\begin{proof}
We will prove the lemma only for the series associated with the positive coordinates of $w^{(n)}$, \emph{i.e.}, we will prove that $\sum_{m=0}^{|w^{(n)}|_+} \ell_m(w^{(n)})$ are uniformly bounded. The proof for the negative part is similar.

Let $1 < A < |\beta|$ such that $\liminf_{n \to \infty} A^n \llbracket \xi - \beta_0^n \tau\rrbracket = \infty$ for every $\xi \in \Psi$, which exists by definition of good direction. Let $\eta = \frac{A-1}{A}|\beta| \in (0, |\beta| - 1)$, which satisfies $\frac{|\beta|}{|\beta|-\eta} = A$. Let $\rho=\frac{\log(|\beta|-\eta)}{\log(\alpha^{-1}+\eta)}>0$. It is sufficient to prove that there exist constants $C_1, C_2 > 0$ such that $\ell_m(w^{(n)}) \leq C_1\exp(-C_2 m^\rho)$ for every $0\leq m \leq |w^{(n)}|_+$ and for every sufficiently large $n$.
This is the same as saying that $\frac{\Re(\gamma_m(w^{(n)})}{m^\rho} \geq C_2 - \frac{\log C_1}{m^\rho}$ for $1 \leq m \leq |w^{(n)}|_+$. To prove this, it is enough to show that
$$\liminf_{n\to\infty} \min_{1\leq m\leq |w^{(n)}|_+} \frac{\Re(\gamma_m(w^{(n)}))}{m^\rho} > 0.$$

We proceed by contradiction and suppose that there exists subsequences of natural numbers $(n_k)_{k \geq 1}$  and  $(m_k)_{k \geq 1}$ such that $0 \leq m_k \leq |w^{(n_k)}|_+$ and
\begin{equation}\label{eq:hypothesis} 
\lim_{k\to\infty} \frac{\Re(\gamma_{m_k}(w^{(n_k)}))}{{m_k}^\rho} = 0. \end{equation}

Let $\bar w^{(n_k)} = S^{m_k}(w^{(n_k)})$ for every $k \geq 1$. Denote by $(p^{(n_k)}_m, c^{(n_k)}_m, s^{(n_k)}_m)_{0\leq m\leq n_k-1}$ and
$(\bar p^{(n_k)}_m, \bar c^{(n_k)}_m, \bar s^{(n_k)}_m)_{0\leq m\leq n_k-1}$ be the finite prefix-suffix decompositions of $w^{(n_k)}$ and $\bar w^{(n_k)}$ respectively.

By taking subsequences if necessary, we can assume that there exist distinct elements $(p,c,s), (\bar p, \bar c, \bar s) \in \A_a$ such that, for $k \geq 1$:
\begin{enumerate}[label=(\roman{*})]
	\item $(p_{n_k - 1}^{(n_k)}, c_{n_k - 1}^{(n_k)}, s_{n_k - 1}^{(n_k)}) = (p, c, s)$;
	\item $(\bar p_{n_k - 1}^{(n_k)}, \bar c_{n_k - 1}^{(n_k)}, \bar s_{n_k - 1}^{(n_k)}) = (\bar p, \bar c, \bar s)$;
	\item $\lim_{k \to \infty} \beta_0^{n_k}\tau = \xi \in \SS^1$.
\end{enumerate}
Now, since $\bar w^{(n_k)}=S^{m_k}(w^{(n_k)})$ with $m_k \geq 0$, we have that
\begin{equation}\label{eq:teo1_1}
\sigma^{n_k - 1}(\bar p_{n_k - 1}^{(n_k)}) \ldots \bar p_0^{(n_k)} = \sigma^{n_k - 1}(p_{n_k - 1}^{(n_k)}) \ldots p_0^{(n_k)} w_0^{(n_k)} \ldots w_{m_k - 1}^{(n_k)}
\end{equation}
for every $k \geq 1$.

We will now reverse the indexes of the finite prefix-suffix decompositions of $w^{(n_k)}$ and $\bar w^{(n_k)}$ in order to obtain sequences in $\S_a $. Let $(x_{n_k})_{k \geq 1}$ and $(y_{n_k})_{k \geq 1}$ be the sequences in $\S_a$ obtained by reversing the coordinates of $(p_m^{(n_k)}, c_m^{(n_k)}, s_m^{(n_k)})_{0 \leq m \leq n_k - 1}$ and $(\bar p_m^{(n_k)}, \bar c_m^{(n_k)}, \bar s_m^{(n_k)})_{0 \leq m \leq n_k - 1}$ and such that $p_m^{x_{n_k}} = p_m^{y_{n_k}}= \varepsilon$ for each $m \geq n_k + 1$.

Without loss of generality we will assume that $x_{n_k}$ converges to $x_\infty \in \S_{a,(p,c,s)}$. By Lemma 5.12 in \cite{CGM}, $x_\infty$ is the representation of an extreme point in $E_a(\tau)$. We will show that any limit point of $y_{n_k}$ in $\S_{a,(\bar p,\bar c,\bar s)}$ is the representation of an extreme point in $E_a(\tau)$ and therefore $\tau$ belongs to $\Psi_a$.
\medbreak

Applying $\gamma$ to \eqref{eq:teo1_1}, using the definitions of $x_{n_k}$ and $y_{n_k}$, and multiplying by $\tau|\beta|^{-n_k}$, we get that for every $k \geq 1$:
$$
\beta_0^{n_k}\tau \sum_{m = 1}^{n_k} \beta^{-m} \Gamma( p_{m}^{y_{n_k}} ) 
= \beta_0^{n_k}\tau \sum_{m = 1}^{n_k} \beta^{-m} \Gamma ( p_{m}^{x_{n_k}} ) + |\beta|^{-n_k} \gamma_{m_k}(w^{(n_k)}).
$$

Let us write $\tau_k=\beta_0^{n_k}\tau$ for $k\geq 1$.
By taking real parts and rearranging the previous expression we obtain:
$$
\Re( \tau_k ( \z_a ( y_{n_k} ) - \z_a ( x_{n_k} ) )) = |\beta|^{-n_k} \Re( \gamma_{m_k}(w^{(n_k)} ).
$$

Furthermore, $\Re(\tau_k \z_a(x_{n_k})) \leq \Re(\tau_k \z_a(y_{n_k}))$ since $w^{(n_k)}$ is minimal. Then we get
\begin{equation}\label{eq:teo1_2}
0 \leq \Re(\beta_0^{n_k}\tau ( \z_a( y_{n_k} )- \z_a(x_{n_k} ) )) = |\beta|^{-n_k} \Re( \gamma_{m_k}(w^{(n_k)} ).
\end{equation}
On the other hand, since $w_0^{(n_k)} \ldots w_{m_k-1}^{(n_k)}$ is a subword of $\sigma^{n_k}(a)$ and $|\sigma^{n_k}(a)|$ grows as $\alpha^{-n_k}$ (recall that $\alpha^{-1}>1$ is the Perron--Frobenius eigenvalue of $M$), we have that $m_k \leq (\alpha^{-1}+\eta)^{n_k}$ for sufficiently large $k \geq 1$.
Therefore, by definition of $\rho$, 
$$
m_k^{-\rho}\geq (\alpha^{-1}+\eta)^{-n_k \rho} = (|\beta|-\eta)^{-n_k} \geq |\beta|^{-n_k}.
$$
From assumption \eqref{eq:hypothesis}, we obtain that
$$
\lim_{k\to \infty} ( |\beta|-\eta )^{-n_k} \Re( \gamma_{m_k}(\omega^{(n_k)} ) )= \lim_{k\to \infty} |\beta|^{-n_k} \Re( \gamma_{m_k}(\omega^{(n_k)} ) )= 0.
$$
In particular, from equation~\eqref{eq:teo1_2} we obtain that any limit point $y_\infty$ of $y_{n_k}$ in $\S_{a,(\bar p,\bar c, \bar s)}$ is such that
$\z_a(y_\infty)$ is an extreme point for the direction $\tau = \lim_{k \to \infty} \beta_0^{n_k}$, that is, $v_a(\tau) = \Re(\tau \z_a(y_\infty))$. Therefore, $\tau$ belongs to $\Psi_a$.

Recall that $A = \frac{|\beta|}{|\beta| - \eta} \in (1, |\beta|)$ is the constant in the definition of good direction. Amplifying equation \eqref{eq:teo1_2} by $A^{n_k}$, we find that
$$
0 \leq A^{n_k} \Re(\tau_k ( \z_a( y_{n_k} )- \z_a(x_{n_k} ) ) ) = (|\beta| - \eta)^{-n_k} \Re( \gamma_{m_k}(\omega^{(n_k)} )
$$
for all sufficiently large $k$. Hence,
\begin{equation}\label{eq:contradicted}
\lim_{k \to \infty} A^{n_k} \Re(\tau_k ( \z_a( y_{n_k} )- \z_a(x_{n_k} ) ) ) = 0.
\end{equation}

Since $x^{(n_k)}$ is minimal,
$$v_a^{(n_k)}(\tau_k) = v_{a,(p,c,s)}^{(n_k)}(\tau_k) = \Re(\tau_k \z_a (x_{n_k})).$$
We also know that $\Re(\beta_0^{n_k} \z_a (y_{n_k})) \geq v_{a,(\bar p,\bar c,\bar s)}^{(n_k)}(\beta_0^{n_k})$ and therefore that
\begin{equation} \label{eq:desig-1}
\Re(\tau_k ( \z_a( y_{n_k} )- \z_a(x_{n_k} ) ) ) \geq v_{a,(\bar p,\bar c,\bar s)}^{(n_k)}(\tau_k) - v_{a,(p,c,s)}^{(n_k)}(\tau_k) \geq 0. \end{equation} 

On the other hand, since $x_\infty\in \S_{a,(p,c,s)}$ and $y_\infty\in\S_{a,(\bar p,\bar c,\bar s)}$ are representations of extreme points for the same direction $\tau$, the unique representation property we are assuming implies
that $\z_a(x_\infty) \neq \z_a(y_\infty)$.

Using Lemma~\ref{lem:converg-exponencial}
we conclude that for each $k \geq 1$:
\begin{equation}\label{eq:min-exponential}
v_{a,(\bar p,\bar c,\bar s)}^{(n_k)}(\tau_k) - v_{a,(p,c,s)}^{(n_k)}(\tau_k) \geq v_{a,(\bar p,\bar c,\bar s)}(\tau_k) - v_{a,(p,c,s)}(\tau_k) - 2C |\beta|^{-n_k}\end{equation}
for a constant $C > 0$ which does not depend on $k$.

Finally, by \eqref{eq:desig-1}, \eqref{eq:min-exponential} and Lemma \ref{lem:central} we conclude that
\begin{equation*}
\Re(\tau_k ( \z_a( y_{n_k} )- \z_a(x_{n_k} ) ) ) \ge
D \llbracket \xi - \tau_k \rrbracket - 2C|\beta|^{-n_k} \end{equation*}
for infinitely many $k \geq 1$. Since $\gamma$ is a good eigenvector and $\xi\in\Psi_a$, by definition, $\liminf_{k\to\infty} A^{n_k} \llbracket \xi - \tau_k \rrbracket = \infty$. This contradicts \eqref{eq:contradicted}.	
\end{proof}

\begin{proof}[Proof of Corollary \ref{cor:finite min points}]
	Let $\omega \in \Omega_{\sigma}$ be a minimal sequence for the vector $\gamma=\tau\Gamma$. We have that $S^n(\omega)$ is also a minimal sequence for some $n \in \Z$ if and only if $\Re(\gamma_n(\omega)) = 0$. By Lemma \ref{lem:cota_minimos}, we have that there exists $n_0\geq1$ 
	such that $\Re(\gamma_n(\omega)) > 0$ for every $n \in \Z$ with $|n| \geq n_0$. This concludes the proof.
\end{proof}

\subsection{Proof of Theorem \ref{thm:no_conjugacy}}

\begin{proof}[Proof of Theorem \ref{thm:no_conjugacy}]
	For $t \in [0, 1)$, let $\iota(t) \in \Omega_\sigma$ be its itinerary by $T$ with respect to the partition $(I_a; a \in \A)$.
	Let $f$ be an affine i.e.m.\ with slope $\ell=(\ell_a; a\in\A)$ which is semi-conjugate to $T$. Then there exists a continuous surjective map $h \colon [0, 1) \to [0, 1)$ such that $h \circ f = T \circ h$. Let $\mu = (\iota \circ h)_*\mathrm{Leb}$ be the pushforward by $\iota \circ h$ of the Lebesgue measure on $[0, 1)$, that is, $\mu(J) = \mathrm{Leb}((\iota \circ h)^{-1}(J))$ for any Borel set $J$ of $\Omega_{\sigma}$. It is easy to see that
	$$\mu(S(J)) = \ell_a \mu(J)$$
	for every $J \subseteq \Omega_a$, where $\Omega_a = \{\omega \in \Omega_\sigma; \omega_0 = a\}$. More generally, if $w = w_{-k} \ldots w_{k'}$ and $\Omega_w = \{\omega \in \Omega_\sigma; \omega_m = w_m, -k \leq m \leq k'\}$, then if $J \subseteq \Omega_w$:
	\begin{equation}\label{eq:prop_cociclo}
	\mu(S^{k'}(J)) = \ell_{k'}(w) \mu(J) \quad \text{and}\quad \mu(S^{-k}(J)) = \ell_{-k}(w) \mu(J).
	\end{equation} 
	
	For each $a\in\A$ and $n \geq 0$, denote $r_{a,n} = |\sigma^n(a)|$ and let $\Omega_{a,n} = \Omega_{\sigma^n(a)}$. We have that $I$ is the union of the sets $S^m(\Omega_{a,n})$ with $a \in \A$ and $0 \leq m \leq r_{a,n} - 1$. Therefore,
	$$
	1 = \mu(\Omega_\sigma) \leq \sum_{a\in\A} \sum_{m=0}^{r_{a,n}-1} \mu(S^m(\Omega_{a,n})).
	$$
	
	We obtain that there exist $\delta>0$, $a\in\A$ and a subsequence $(n_k)_{k \geq 1}$ of the natural numbers such that, for every $k \geq 1$,
	\begin{equation}\label{eq: greater_than_delta}
	\sum_{m=0}^{r_{a, n_k}-1} \mu(S^m(\Omega_{a,n_k})) \geq \delta.
	\end{equation}

	Now assume by contradiction that there exists an affine i.e.m.\ $f$ with slope vector
	$\ell=\exp(-\Re(\gamma))$ which is conjugate to $T$. That is, we assume that $h$ is injective. We will show that this contradicts inequality \eqref{eq: greater_than_delta}.

	Let $(w^{(n_k)})_{k \geq 1}$ be a sequence of locally minimal words for $\sigma^{n_k}(a)$ and the vector $\gamma$, \emph{i.e.}, there exists $0\leq m_k \leq r_{a, n_k}-1$ such that	
	$$w^{(n_k)}=S^{m_k}(\sigma^{n_k}(a)) = w^{(n_k)}_{-m_k} \ldots w^{(n_k)}_{-1} . w^{(n_k)}_{0} \ldots w^{(n_k)}_{r_{a,n_k}-1-m_k}$$
	with $\Re(\gamma_n(w^{(n_k)}))\geq 0$ for every $-m_k\leq n \leq r_{a,n_k}-1-m_k$. By Lemma~\ref{lem:cota_minimos}, there exists a constant $K>0$ such that $\Sigma(w^{(n_k)}, \ell) < K$ for every $k \geq 1$. Note that $\Omega_{w^{(n_k)}}=S^{m_k}(\Omega_{a,n_k})$ and $S^m(\Omega_{a,n_k}) = S^{m-m_k} (S^{m_k}(\Omega_{a,n_k}))=S^{m-m_k}(\Omega_{w^{(n_k)}})$. Thus, by equation \eqref{eq:prop_cociclo}, 
	$$\mu(S^m(\Omega_{a,n_k})) = \ell_{m-m_k} (w^{(n_k)}) \mu(\Omega_{w^{(n_k)}}).$$
	Therefore,
	\begin{align*}
	\sum_{m=0}^{r_{a,n_k}-1} \mu(S^m(\Omega_{a,n_k})) &= \mu(\Omega_{w^{(n_k)}}) \sum_{m=0}^{r_{n_k}-1} \ell_{m-m_k}(w^{(n_k)}) \\
	&= \mu( \Omega_{w^{(n_k)}}) \Sigma(w^{(n_k)}, \ell)
	 \leq K \mu( \Omega_{w^{(n_k)}})
	\end{align*}
	for all $k \geq 1$. Finally, let $\omega \in \Omega_{\sigma}$ be a limit point of $(w^{(n_k)})_{k \geq 1}$. We have that $\liminf_{k \to \infty} \mu(\Omega_{w^{(n_k)}}) \leq \mu(\{\omega\})$. Since $h$ is invertible and $\iota$ is injective, we have that $\mu(\{\omega\}) = 0$, so
	$$\liminf_{k \to \infty} \sum_{m=0}^{r_{a,n_k}-1} \mu(S^m(\Omega_{a,n_k})) = 0,$$
	contradicting \eqref{eq: greater_than_delta}.
	
\end{proof}

\subsection{Proof of Theorem~\ref{thm:no-wandering}} We start by showing that minimal components of $H$ contain directions in $\Psi$.

\begin{lem}\label{lem:}
	If $Y$ is a minimal component of $H$, then there exists $(\xi, (q,a,r)) \in Y$ such that $\xi \in \Psi_a$.
\end{lem}

\begin{proof}
	Assume that for each $(\xi, (q,a,r)) \in Y$ there exists a unique $(p,c,s) \in \bA_a$ with $v_a(\xi) = v_{a,(p,c,s)}(\xi)$. Let $(\xi, (q,a,r)) \in Y$ and let $Z \subseteq Y$ be the maximal interval containing $(\xi, (q,a,r))$. That is, $Z = J \times \{(q,a,r)\}$, where $J \subseteq \SS^1$ is the maximal interval containing $\xi$ such that $Z \subseteq Y$. We will prove that $J = \SS^1$.
	
	By  hypothesis, $H^m(Z) = \beta_0^{-m}J \times \{(p_m,c_m,s_m)\}$ for some $(p_m,c_m,s_m) \in \bA$ and each $m \geq 0$. By minimality, there exists $n \geq 1$ such that
	$H^n(Z) \cap Z \neq \varnothing$, which implies that $(p_n,c_n,s_m) = (q,a,r)$. We obtain that:
	$$
	Z \subseteq Z \cup H^n(Z) = (J \cup \beta_0^{-n}J) \times \{(q,a,r)\} \subseteq Y.
	$$
	If $J$ was a proper subset of $\SS^1$, then it would also be a proper subset of $J \cup \beta_0^{-n}J$. This contradicts the maximality of $J$. Therefore, $Z = \SS^1 \times \{(p,a,s)\}$.
	
	Now, by minimality, there exists $n \geq 0$ such that $\bigcup_{m = 0}^n H^m(Z) = Y$. Assume that $n$ is minimal with this property. By hypothesis, we deduce that:
	$$
	Y = \SS^1 \times \{(p_m,c_m,s_m); 0 \leq m \leq n\}
	$$
	and that $H(\tau, (p_m,c_m,s_m)) = (\beta_0^{-1}\tau, (p_{m+1}, c_{m+1}, s_{m+1}))$ for every $m \geq 0$ and $\tau \in \SS^1$. We conclude that the projection $\pi_{\bA}$ of any pre-orbit by $H$ is periodic. Fixing a good eigenvector $\gamma$, by Theorem \ref{thm:minimal points} we obtain that there exist ultimately periodic minimal sequences for $\gamma$. This contradicts Lemma 5.8 in \cite{CGM}.
\end{proof}

\begin{lem}
	Under the hypotheses of Theorem~\ref{thm:enter minimal}, if $\gamma=\tau\Gamma$ is a very bad eigenvector, $\omega$ is a minimal sequence for $\gamma$ and $\ell=\exp(-\Re(\gamma))$, then $\Sigma(\omega, \ell) = \infty$.
\end{lem}

\begin{proof}
	Let $(p_m,c_m,s_m)_{m\geq 0}$ be the prefix-suffix decomposition of $\omega$. We claim that there exist a subsequence $(n_k)_{k\geq 1}$ and distinct labels $(p,c,s)$ and $(\bar p, \bar c, \bar s)$ in $\bA_a$ such that, for every $k \geq 1$,
	\begin{enumerate}[label=(\roman{*})]
		\item $c_{n_k} =a$;
		\item $\tau_k = \beta_0^{n_k}\tau \to \xi$ when $k \to \infty$;
		\item $(p_{n_k-1}, c_{n_k-1}, s_{n_k-1}) = (p,c,s)$;
		\item $v_a(\xi) = v_{a, (p,c,s)}(\xi) = v_{a, (\bar p, \bar c, \bar s)}(\xi)$.
	\end{enumerate}	

	From these conditions we can prove that $\Sigma(\omega,\ell)=\infty$. 
	Indeed,	let $x_k$ be an element of $\S_a$ whose first $n_k$ coordinates coincide with $(p_{n_k-m}, c_{n_k-m}, s_{n_k-m})_{1 \leq m \leq n_k}$. Clearly, $x_k \in \S_{a,(p,c,s)}$.
	
	By definition of $x_k$ and the fact that $\omega$ is minimal, we obtain from Lemma \ref{lem:minimalprefix} that $v_a^{(n_k)}(\tau_k)=v_{a,(p,c,s)}^{(n_k)}(\tau_k) =\Re(\tau_k \z_a^{(n_k)}(x_k))$. 
	
	Let $y_k \in \S_{a,(\bar p,\bar c,\bar s)}$ be such that
	$v_{a,(\bar p, \bar c, \bar s)}^{(n_k)}(\tau_k) = \Re( \tau_k \z_a(y_k))$ for every $k \geq 1$.
	Clearly, $v_a^{(n_k)}(\tau_k) \leq v_{a,(\bar p, \bar c, \bar s)}^{(n_k)}(\tau_k)$ and from Lemma~\ref{lem:converg-exponencial} and Lemma~\ref{lem:central}, there exists $C > 0$ such that, for every sufficiently large $k$,
	\begin{align*}
	0\leq v_{a, (\bar{p}, \bar{c},\bar{s})}^{(n_k)}(\tau_k) - v_{a, ({p}, {c},{s})}^{(n_k)}(\tau_k ) & \leq v_{a, (\bar{p}, \bar{c},\bar{s})}(\tau_k)-v_{a, ({p}, {c},{s})}(\tau_k ) + 2C|\beta|^{-n_k} \\
	& \leq |\beta|^{-n_k} (2 D_2|\beta|^{n_k} \llbracket \xi - \tau_k \rrbracket + 2C)
	\end{align*}
	and therefore, from the fact that $\tau$ is a very bad direction, we conclude that, for some constant $C_2 > 2C$ and every sufficiently large $k$,
	\begin{equation}\label{eq:constant_C_2}
	\Re(\tau_k\z_a(y_{k})) -\Re(\tau_k\z_a( x_{k})) = v_{a, (\bar{p}, \bar{c},\bar{s})}^{(n_k)}(\tau_k) - v_{a, ({p}, {c},{s})}^{(n_k)}(\tau_k ) \leq C_2|\beta|^{-n_k}.
	\end{equation} 
	
	If $\bar p \bar c$ is a prefix of $p$, then
	$$\sigma^{n_k-1}(p^{x_k}_1)\dots p^{x_k}_{n_k-1} = \sigma^{n_k-1}(p^{y_k}_1)\dots p^{y_k}_{n_k-1} \omega_0\omega_1\dots\omega_{m_k-1}$$
	for some increasing sequence $(m_k)_{k \geq 1}$. Therefore,
	$$\Re(\gamma_{m_k}(\omega)) = |\beta|^{n_k}\Re(\tau_k\z_a(x_{k})) -|\beta|^{n_k} \Re(\tau_k\z_a(y_{k})) \leq 0$$
	and, since $\omega$ is a minimal sequence, we obtain that $\Re(\gamma_{m_k}(\omega)) = 0$ for every $k \geq 1$. We conclude that $\ell_{m_k}(\omega) = \exp(-\Re(\gamma_{m_k}(\omega))) = 1$, so $\Sigma(\omega, \ell) = \infty$.
	
	If $pc$ is a prefix of $\bar p$, then
	$$\sigma^{n_k-1}(p^{y_k}_1)\dots p^{y_k}_{n_k-1} = \sigma^{n_k-1}(p^{x_k}_1)\dots p^{x_k}_{n_k-1} \omega_0\omega_1\dots\omega_{m_k-1}$$
	for some increasing sequence $(m_k)_{k \geq 1}$. Therefore,
	$$\Re(\gamma_{m_k}(\omega)) = |\beta|^{n_k}\Re(\tau_k\z_a(y_{k})) -|\beta|^{n_k} \Re(\tau_k\z_a(x_{k}).$$
	By \eqref{eq:constant_C_2}, we obtain that
	$$\Re(\gamma_{m_k}(\omega)) = |\beta|^{n_k}\Re(\tau_k\z_a(y_{k})) -|\beta|^{n_k} \Re(\tau_k\z_a( x_{k})) \leq C_2$$
	for all sufficiently large $k$. We conclude that $\ell_{m_k}(\omega) = \exp(-\Re(\gamma_{m_k}(\omega))) \geq \exp(-C_2)$, so $\Sigma(\omega, \ell) = \infty$.
	
	We will now prove that we can find a sequence $(n_k)_{k \geq 1}$ such that (i)-(iv) hold. We consider two complementary cases:
	
	\noindent{\bf Case 1.}
	Assume that there exists $m_0 \geq 1$ such that, for all $m \geq m_0$,
	$$ H(\beta_0^{m}\tau, (p_{m}, c_{m}, s_{m})) = (\beta_0^{m-1}\tau, (p_{m-1}, c_{m-1}, s_{m-1})).$$
	That is, up to finitely many terms, the prefix-suffix decomposition of $\omega$ is obtained by the projection $\pi_{\bA}$ of a pre-orbit by $H$. By Lemma~\ref{lem:minimalarrival} we have that $(\beta_0^{m_0}\tau, (p_{m_0}, c_{m_0}, s_{m_0}))$ belongs to a minimal component $Y$ of $H$ and, therefore, that
	$(\beta_0^{m}\tau, (p_{m}, c_{m}, s_{m}))$ belongs to $Y$ for all $m \geq m_0$.
	
	By the previous lemma, there exists $(\xi, (q, a, r)) \in Y$ such that $\xi \in \Psi_a$.
	Since $H^{-1}$ is minimal, we can find a sequence
	$(n_k)_{k\geq 1}$ such that $(\beta_0^{n_k}\tau, (p_{n_k}, c_{n_k}, s_{n_k}))$ converges to $(\xi, (q, a, r))$. We can also assume that $(p_{n_k-1}, c_{n_k-1}, s_{n_k-1}) = (p,c,s)$ is constant. We then obtain claims (i)-(iii).
	
	By definition of $H$, we have that $v_a(\tau_k) = v_{a,(p,c,s)}(\tau_k)$ and, by continuity, we obtain that $v_a(\xi) = v_{a,(p,c,s)}(\xi)$. Since $\xi \in \Psi_a$, there exists $(p,c,s) \neq (\bar p, \bar c, \bar s) \in \bA_a$ such that $v_a(\xi) = v_{a, (\bar p, \bar c, \bar s)}(\xi)$. In this way, we obtain (iv), and the claim holds in this case.
	
	\noindent{\bf Case 2.}
	Assume that there exists a subsequence $(p_{n_k}, c_{n_k}, s_{n_k})_{k\geq 1}$ such that
	$$ H(\beta_0^{n_k+1}\tau, (p_{n_k+1}, c_{n_k+1}, s_{n_k+1})) 
	\neq (\beta_0^{n_k}\tau, (p_{n_k}, c_{n_k}, s_{n_k}))$$
	for every $k \geq 1$.
	Without loss of generality, we may assume that there exists $a\in\A$ and distinct labels $(p,c,s), (\bar p, \bar c, \bar s) \in \bA_a$ such that, for all $k\geq 1$,
	\begin{enumerate}[label=(\roman{*})]
		\item $c_{n_k} =a$;
		\item $\tau_k = \beta_0^{n_k}\tau \to \xi$ when $k \to \infty$;
		\item $(p_{n_k-1}, c_{n_k-1}, s_{n_k-1}) = (p,c,s)$;
	\end{enumerate}	
	and $H(\beta_0^{n_k}\tau, (p_{n_k}, c_{n_k}, s_{n_k})) = (\beta_0^{n_k-1}\tau, (\bar p, \bar c, \bar s))$.
	
	We therefore have conditions (i)-(iii) of our claim, and we must prove that (iv) holds.

	Notice that, by continuity of $v_a$ and definition of $H$, $v_a(\xi) =v_{a, (\bar p, \bar c, \bar s)}(\xi)$. To obtain (iv), it is enough to show that $v_a(\xi)=v_{a, (p, c, s)}(\xi)$.
	
	Let $x_k \in \S_{a}$ such that its first $n_k$ coordinates are $(p_{n_k-m}, c_{n_k-m}, s_{n_k-m})_{1 \leq m \leq n_k}$. Clearly, $x_k \in \S_{a,(p,c,s)}$. By definition of $x_k$ and the fact that $\omega$ is minimal, we obtain from Corollary \ref{cor:minimalprefix} that $v_a^{(n_k)}(\tau_k)=v_{a,(p,c,s)}^{(n_k)}(\tau_k) =\Re(\tau_k \z_a^{(n_k)}(x_k))$. 
	Therefore, if $x$ is any limit point of in $\S_{a,(p,c,s)}$ of $(x_k)_{k \geq 1}$, then $v_a(\xi) =\Re(\xi \z_a({x}))$. Thus, $\xi\in\Psi_a$. We conclude that our claim also holds in the second case.
	
\end{proof}

\begin{proof}[Proof of Theorem~\ref{thm:no-wandering}]
Consider an affine extension $f$ of $T$ with slope vector $\ell = \exp(-\Re(\gamma))$. Let $h\colon I \to I$ be a continuous, surjective and non-decreasing map such that $h \circ f = T \circ h$.
We claim that if $h$ is not injective, then there exists a minimal sequence $\omega \in \Omega_{\sigma}$ for which the series $\Sigma(\omega,\ell)$ is finite. 

Indeed, let $t_0 \in I$ be such that $J_0 = h^{-1}(\{t_0\}) \subseteq I$ is an interval of positive length (that is, $J_0$ is a wandering interval for $f$). Let $\omega \in \Omega_{\sigma}$ be the itinerary of $t_0$ by $T$. Observe that, since $h$ is non-decreasing, for any $n \in \Z$:
$$
f^n(J_0) = f^n(h^{-1}(t_0)) = h^{-1}(T^n(t_0))
$$
so $f^n(J_0) \subseteq h^{-1}(I_{\omega_n})$ for every $n \in \Z$. Since, for every $a \in \A$, the slope of $f$ on the interval $h^{-1}(I_a)$ is $\ell_a$, we obtain that, for every $n \in \Z$,
$$
\mathrm{Leb}(f^n(J_0)) = \ell_n(\omega) \mathrm{Leb}(J_0),
$$
where $\mathrm{Leb}$ is the Lebesgue measure on $I$. The intervals $f^n(J_0)$ are pairwise disjoint, so
$$
\sum_{n \in \Z} \mathrm{Leb}(f^n(J_0)) = \sum_{n \in \Z} \ell_n(\omega) \mathrm{Leb}(J_0) = \Sigma(\omega, \ell) \mathrm{Leb}(J_0) \leq 1,
$$
which shows that $\Sigma(\omega, \ell) \leq \frac{1}{\mathrm{Leb}(J_0)} < \infty$. The sequence $\omega$ must then be minimal up to a shift. The proof follows by the previous lemma.
\end{proof}
\section{The cubic Arnoux--Yoccoz map} \label{sec:example}

In the cubic Arnoux--Yoccoz i.e.m.\ (A-Y i.e.m.)\ we illustrate the main theorems of the article. In particular, we construct the map $H$ together with its minimal components. We have to mention that this example is not really self-similar in the sense of this article, but the natural symbolic coding is substitutive and the substitution satisfies the conditions of this article. In any case, it can be transformed in such a way that the resulting i.e.m.\ fully satisfies our conditions, but the extra notation is unnecessary to understand the phenomenon. Details on this transformation can be found in \cite{cubicAY}. 

Let $\alpha$ be the unique real number such that $\alpha + \alpha^2 + \alpha^3 = 1$ and let $G_{t_0,t_1}$ be the map exchanging both halves of the interval $[t_0, t_1)$ while preserving orientation. That is,
$$
G_{t_0, t_1}(t) = \begin{cases}
	t + (t_0 + t_1)/2 & t \in [t_0, (t_0+t_1)/2), \\
	t - (t_0 + t_1)/2 & t \in [(t_0+t_1)/2, t_1), \\
	t & t \notin [t_0, t_1).
\end{cases}
$$
Then, the A-Y i.e.m.\ is given by $T = G_{0,1} \circ G_{0,\alpha} \circ G_{\alpha, \alpha+\alpha^2} \circ G_{\alpha+\alpha^2,1}$. Properties of $T$ were extensively discussed in \cite{geometricalmodels}. In particular, it is proved that the map $T$ is equal, up to rescaling and rotation, to the map induced on the interval $[0, \alpha)$ and, by considering an appropriate refinement of continuity intervals of $T$ into nine intervals, one may encode the relation of orbits by $T$ for this partition and the orbits of the induced system for the induced partition by the following substitution $\sigma$ on the alphabet $\A = \{1, \ldots, 9\}$:
$$\arraycolsep=1.5pt
	\begin{array}{lllllllll}
		\sigma(1) & = & 35 &\qquad \sigma(4) & = & 17 &\qquad \sigma(7) & = & 29 \\
		\sigma(2) & = & 45 &\qquad \sigma(5) & = & 18 &\qquad \sigma(8) & = & 2 \\
		\sigma(3) & = & 46 &\qquad \sigma(6) & = & 19 &\qquad \sigma(9) & = & 3 \\
	\end{array}
$$
One then has that $\Omega_T = \Omega_\sigma$. It is easy to check that $\sigma$ is primitive.
The characteristic polynomial of $M$ is $(1-t^3)(t^3+t^2+t-1)(-t^3+t^2+t+1)$, where the last two factors are irreducible. The roots of $t^3+t^2+t-1$ are $\alpha$, $\beta$ and $\bar{\beta}$, whereas the roots of $-t^3+t^2+t+1$ are $\alpha^{-1}$, $\beta^{-1}$ and $\bar{\beta}^{-1}$, where $\alpha^{-1}$ is the Perron--Frobenius eigenvalue. We assume that $\beta$ is the eigenvalue with positive imaginary part. Numerically, 
$\beta \approx -0.771845 + 1.11514i$. It is proved in \cite{messaoudi} that $(\beta^{-1})^n$ is never real for any $n \in \Z$. Furthermore, we have that the eigenspace associated with $\beta$ has dimension 1. In fact, it is generated by 
$$\Gamma =
(\beta^2+\beta+1, -\beta, -\beta, -\beta^2-\beta-1, \beta+1, \beta + 1, -\beta^2-\beta-2, -1, -1).$$

In what follows $\beta$ and $\Gamma$ are the corresponding eigenvalue and eigenvector of $M$ used in previous sections. By Lemma 8.8 in \cite{CGM} we have that this example satisfies the u.r.p.\ for the selected $\beta$ and $\Gamma$. Also, the boundaries of the associated fractals 
are Jordan curves (see Lemma 8.5 and Corollary 8.7 in \cite{CGM}).

First we compute the sets $\Psi_a$ where we can find extreme points in different sub-fractals. 
Since $\F_2 = \F_3$, $\F_5 = \F_6$ and $\F_8 = \F_9$ (proved in \cite{geometricalmodels}), we focus on the subalphabet $\{1,2,4,5,7,8\}$.

\begin{lem}
One has
$$\arraycolsep=1.5pt
\begin{array}{llllll}
	\Psi_1 & = & \{ -i \beta_0^4, \eta_1 \} &\qquad \Psi_5 & = & \{ i \beta_0^5, -\beta_0 \eta_1 \} \\
	\Psi_2 & = & \{i \beta_0^2, -i \beta_0^2\} &\qquad \Psi_7 & = & \{ i \beta_0^5, \eta_7 \} \\
	\Psi_4 & = & \left\{ i \beta_0^5, i \frac{ |\beta^{-1} + \beta^{-3}| }{ \beta^{-1} + \beta^{-3} } \right \} &\qquad \Psi_8 & = & \varnothing
\end{array},
$$
where $\phi - \frac{\pi}{2} < \arg(\eta_1) < 2\phi - \frac{3\pi}{2}$ and $4\phi - \frac{9\pi}{2} <\arg(\eta_7) < 2\pi$, $\phi = \arg(\beta_0) \in [0, 2\pi)$. 
\end{lem}
\begin{proof}
The bounds on $\eta_1$ and $\eta_7$ were found computationally and then proved analytically. 
The proof is tedious but elementary, so we omit it here. See the Appendix for computations. 
\end{proof}

Using the previous lemma it is possible to compute the right-continuous map $H$ somewhat explicitly (it will depend on the bounds for $\eta_1$ and $\eta_7$). 
The bounds on $\eta_1$ and $\eta_7$ are sufficiently good so that $H$ restricted to its minimal components does not depend on their exact values.

\begin{figure}
\centering
\includegraphics[scale=0.30]{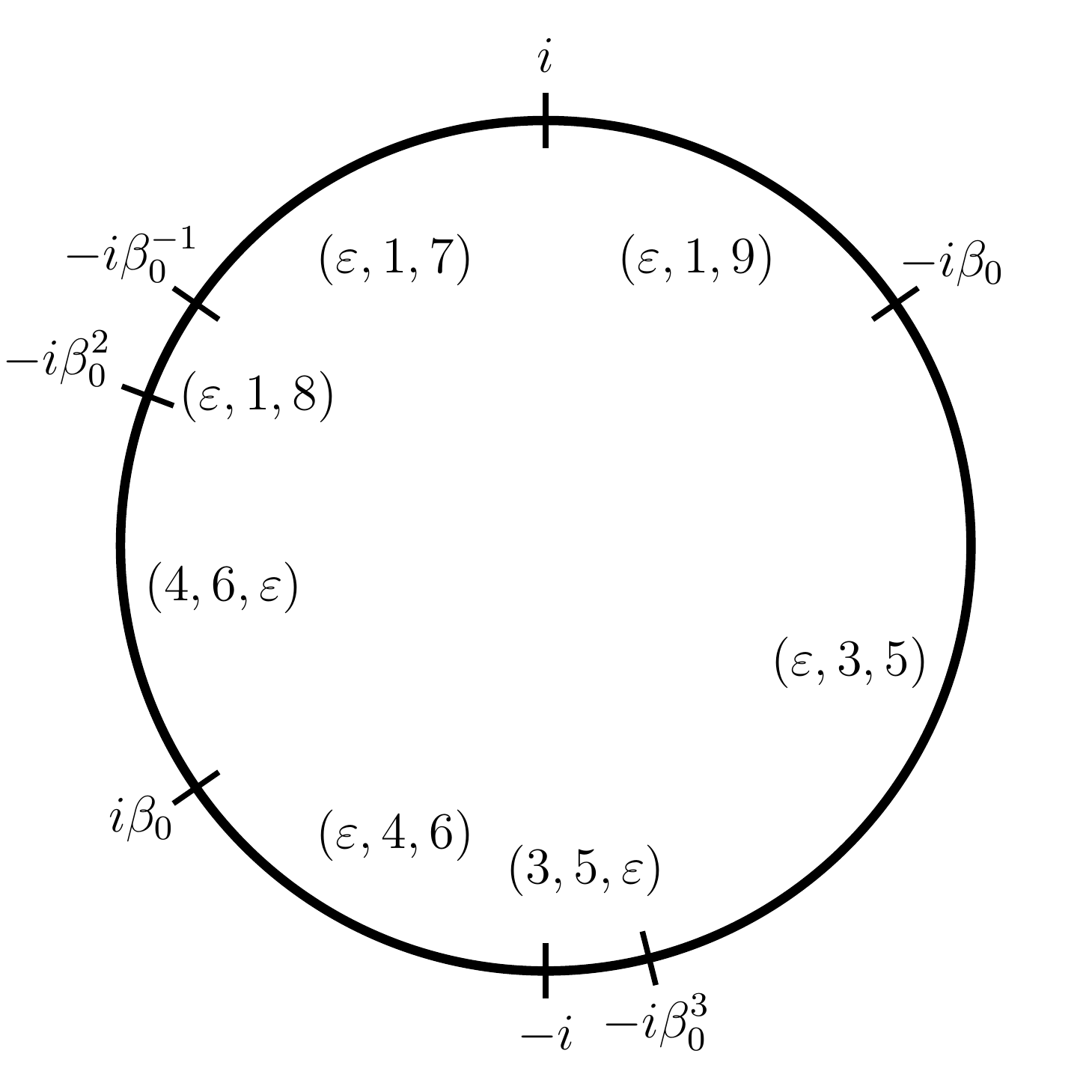} \includegraphics[scale=0.30]{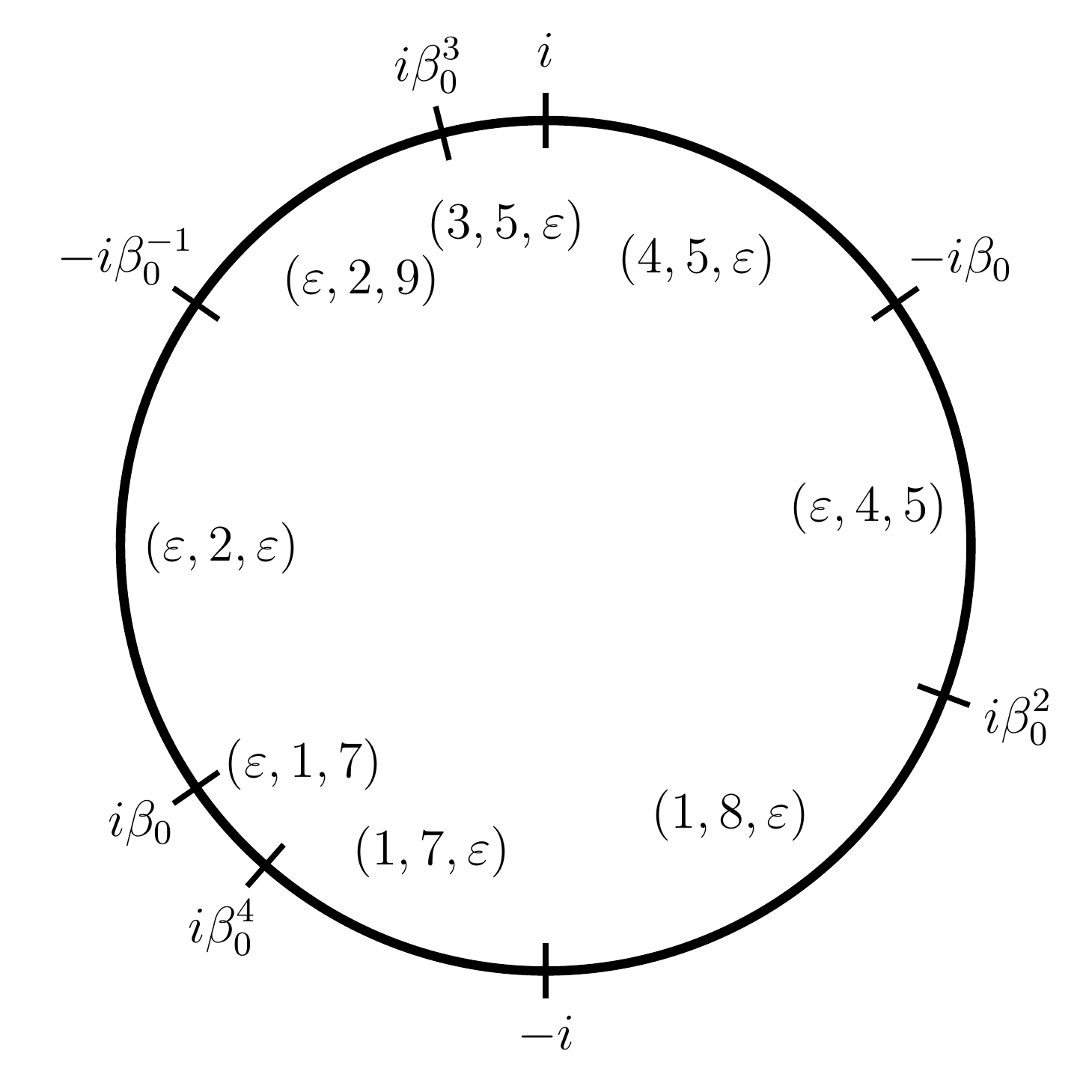}
\caption{Each circle represents a minimal component for the map $H$ associated with the cubic Arnoux--Yoccoz map. The map $H$ acts as a rotation by $\beta_0^{-1}$ in each circle and the labels in $\bA$ change in accordance with the partitions of each circle.}
\label{fig:minimalcomponents}
\end{figure}

\begin{lem}
The map $H$ has exactly two minimal components shown in Figure \ref{fig:minimalcomponents}. 
\end{lem}
\begin{proof}
By iterating $H$, we see that $H^{30}(\SS^1\times\bA) = H^{31}(\SS^1\times\bA)$, so the limit set $\Lambda_H = \bigcap_{n\geq 0} H^n(\SS^1\times\bA)= H^{30}(\SS^1\times\bA)$. In addition, the restriction of $H$ to $\Lambda_H$ coincides with the map in Figure \ref{fig:minimalcomponents}. 
\end{proof}

\begin{lem}
Let $A > 1$ be a real number and let $\tau, \xi \in \SS^1$ be algebraic. Then, $\liminf_{n \to \infty} A^n \llbracket \xi - \beta_0^n \tau \rrbracket = \infty$.
\end{lem}

\begin{proof}
We will use Baker's theorem, which relies on the following definition: given an algebraic number whose minimal primitive polynomial is $p(t) = \sum_{m = 0}^d p_m t^m$, we define its \emph{height} as $\max_{0 \leq m \leq d} |p_m|$.

Let $n \geq 0$. First observe that $\llbracket \xi - \beta_0^n \tau \rrbracket = \llbracket \xi\tau^{-1} - \beta_0^n\rrbracket$ and that $\xi\tau^{-1}$ is algebraic. 
Let $\phi, \theta \in [0, 2\pi]$ such that $\exp(i\phi) = \beta_0$ and $\exp(i\theta) = \xi\tau^{-1}$. We have that $2\pi$ and $\phi$ are linearly independent over the rational numbers, since $\beta_0$ is not a root of unity. Moreover,
$$
\llbracket \xi - \beta_0^n\tau \rrbracket = \llbracket \xi\tau^{-1} - \beta_0^n\rrbracket = \min_{m \in \Z} |2m\pi + \theta - n\phi|
$$
and, since $\theta, \phi \in [0, 2\pi]$, the number $m \in \Z$ attaining the minimum in the previous equation has absolute value at most $n$. Therefore,
$$\llbracket \xi - \beta_0^n\tau \rrbracket = \min_{-n \leq m \leq n} |2m\pi + \theta - n\phi|.$$
We consider several cases:

If $\xi\tau^{-1}$ is neither a root of unity nor a rational power of $\beta_0$, then $2\pi$, $\phi$ and $\theta$ are linearly independent over the rational numbers. The height of $m$ is $|m| \leq n$. By Baker's theorem, we obtain that $|2m\pi + \theta - n\phi| \geq n^{-C}$ for every $-m \leq m \leq n$, where $C > 0$ is a constant independent of $n$. Therefore, $\llbracket \xi - \beta_0^n \tau \rrbracket \geq n^{-C}$.

If $\xi\tau^{-1}$ is a root of unity, then $\theta = 2 q\pi$ with $q$ a rational number. Therefore,
$$\llbracket \xi - \beta_0^n\tau \rrbracket = \min_{-n \leq m \leq n} |2(m + q)\pi - n\phi|.$$
If $n$ is larger than both the numerator and denominator of $q$, then the height of $m + q$ most $n^2 + n$ for every $-n \leq m \leq n$. Baker's theorem then shows that $|2(m + q)\pi - n\phi| \geq (n^2 + n)^{-C}$ for every $-n \leq m \leq m$ with $m \neq -q$, where $C > 0$ is a constant independent of $n$. Observe that the minimum cannot be attained at $m = -q$ except for finitely many $n$. Therefore, $\llbracket \xi - \beta_0^n \tau \rrbracket \geq (n^2 + n)^{-C}$ for sufficiently large $n$.

If $\xi\tau^{-1}$ is a rational power of $\beta_0$, then $\theta = q\phi$ for some rational number $q$. Therefore,
$$\llbracket \xi - \beta_0^n\tau \rrbracket = \min_{-n \leq m \leq n} |2m\pi + (q - n)\phi|.$$
If $n$ is larger than the numerator and denominator of $q$, then the height of $q - n$ is at most $n^2 + n$. By Baker's theorem, we obtain that $|2m\pi + (q - n)\phi| \geq (n^2 + n)^{-C}$ for every $-n \leq m \leq m$ and $n \neq q$, where $C > 0$ is a constant independent of $n$. Therefore, $\llbracket \xi - \beta_0^n \tau \rrbracket \geq (n^2 + n)^{-C}$ for sufficiently large $n$.

In any case, for $n$ sufficiently large, $\llbracket \xi - \beta_0^n\tau \rrbracket$ is bounded from below by $p(n)^{-C}$, where $p$ is a polynomial and $C > 0$ is a constant independent of $n$. This fact rules out the exponential rate of convergence.

\end{proof}

\section{Possible additional examples}

In this section we present a class of i.e.m.'s for which we think it is possible to find new examples verifying the hypotheses of our main results. These hypotheses are: the existence of a suitable eigenvalue, that is, a non-real expanding eigenvalue $\beta$ such that $\beta/|\beta|$ is not a root of unity, and the unique representation property. The family of examples satisfies the first hypothesis. Nevertheless, to determine if the unique representation property holds for a specific example in this class one needs to understand the topology of the associated fractals. We expect that algebraic conditions similar to the ones in \cite{persistence} are sufficient. Indeed, conditions of this nature should imply that the fractals are ``well-behaved'' in a broad sense since this is true for classical Rauzy fractals.

We will make use of the notion of Rauzy--Veech algorithm and related concepts such as Rauzy classes. For more details on these notions we suggest \cite{viana} and \cite{yoccoz-pisa}.

\subsection{Suitable eigenvalue hypothesis} We will restrict the discussion to i.e.m.'s which are periodic for the Rauzy--Veech algorithm. This is a natural class of self-similar i.e.m.'s as explained in \cite[Section 7.2]{CGM}.

First observe that i.e.m.'s exchanging five intervals or less cannot satisfy this hypothesis. Indeed, any reciprocal quintic polynomial of a primitive matrix has at least three real roots. If the remaining roots $\beta$, $\beta^{-1}$ are non-real, then $|\beta| = 1$ as they are complex conjugates.

However, it is possible to construct an infinite family of self-similar i.e.m.'s exchanging six intervals whose induction matrices have suitable eigenvalues by finding appropriate cycles in a Rauzy class as done in \cite[Section 6]{persistence}. Indeed, consider the hyperelliptic permutation
\[
	\pi = \begin{pmatrix}
		1 & 2 & 3 & 4 & 5 & 6 \\ 6 & 5 & 4 & 3 & 2 & 1
	\end{pmatrix}.
\]
We consider three cycles on the Rauzy class of $\pi$ (see Figure \ref{fig:hyperelliptic_rauzy}):
\begin{enumerate}
	\item alternating top and bottom operations until coming back to $\pi$;
	\item alternating bottom and top operations until coming back to $\pi$;
	\item three bottom operations, followed by $2n$ top operations and two more bottom operations, for an integer $n \geq 0$.
\end{enumerate}

\begin{figure} \label{fig:hyperelliptic_rauzy}
\centering
\includegraphics[width=0.7\textwidth]{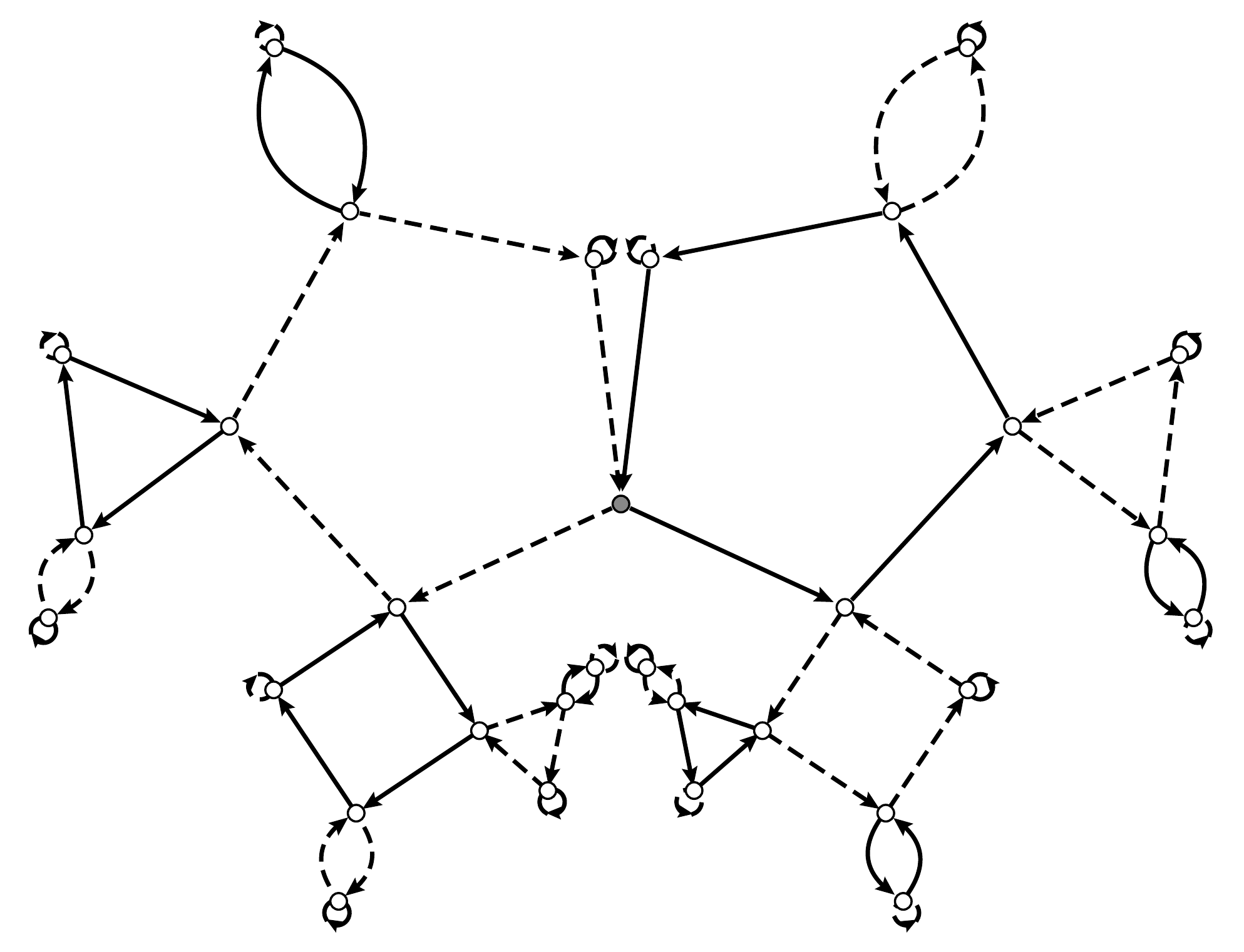}
\caption{The Rauzy class of $\pi$, which is marked with a grey dot. Solid arrows represent top operations, while dashed arrows represent bottom operations. Observe that the first cycle traverses the entire right ``half'' of the Rauzy class, while the second cycle traverses the entire left ``half''.}
\label{fig:rauzyclass}
\end{figure}

The induction matrices $M_1$, $M_2$ and $M_3$ obtained from these three cycles are, respectively, the following:
{\small \[
	M_1 = \begin{pmatrix}
		1 & 0 & 0 & 0 & 0 & 1 \\
		0 & 2 & 10 & 10 & 5 & 1 \\
		0 & 7 & 54 & 54 & 28 & 1 \\
		0 & 22 & 156 & 161 & 84 & 1 \\
		0 & 42 & 298 & 306 & 162 & 1 \\
		0 & 26 & 185 & 190 & 100 & 1
	\end{pmatrix}, \quad
	M_2 = \begin{pmatrix}
		1 & 100 & 190 & 185 & 26 & 0 \\
		1 & 162 & 306 & 298 & 42 & 0 \\
		1 & 84 & 161 & 156 & 22 & 0 \\
		1 & 28 & 54 & 54 & 7 & 0 \\
		1 & 5 & 10 & 10 & 2 & 0 \\
		1 & 0 & 0 & 0 & 0 & 1
	\end{pmatrix}
\]
\[
	M_3 = \begin{pmatrix}
		1 & 0 & n & 0 & 0 & 0 \\
		1 & 1 & 2n & 0 & 0 & 0 \\
		1 & 0 & 1+n & 0 & 0 & 0 \\
		1 & 0 & 0 & 1 & 0 & 0 \\
		1 & 0 & 0 & 0 & 1 & 0 \\
		1 & 0 & 0 & 0 & 0 & 1
	\end{pmatrix}.
\]}

Let $M = M_1M_3M_2$. We will show that $M$ has a non-real expanding eigenvalue.

A straightforward computation shows that its characteristic polynomial is given by
\begin{align*}
	p(t) &= 1 - (196351 + 51729 n)t + (740715 + 183764 n)t^2 - (1092962 +
	 269314 n) t^3 \\
	 &\quad + (740715 + 183764 n)t^4 - (196351 + 51729 n)t^5 + t^6.
\end{align*}

Let $\alpha^{-1}$ be the Perron--Frobenius root of $p$ and $s = \alpha + \alpha^{-1}$. Let $\beta$, $\beta^{-1}$, $\gamma$, $\gamma^{-1}$ be the other four roots of $p$ and $u = \beta + \beta^{-1}$, $v = \gamma + \gamma^{-1}$. By expanding the equality $p(t) = (t - \alpha)(t - \alpha^{-1})(t - \beta)(t - \beta^{-1})(t - \gamma)(t - \gamma^{-1})$, we obtain that:
\begin{itemize}
	\item $s + u + v = 196351 + 51729 n$;
	\item $3 + s u + u v + v s = 740715 + 183764 n$;
	\item $s u v + 2(s + u + v) = 1092962 +
		 269314 n$.
\end{itemize}
This can be reduced to
\begin{itemize}
	\item $s + u + v = 196351 + 51729 n$;
	\item $s u + u v + v s = 740712 + 183764 n$;
	\item $s u v = 700260 + 165856 n$.
\end{itemize}
Therefore, we have that
\begin{align*}
	q(t) &= (t - s)(t - u)(t - v) \\
	&= -(700260 + 165856 n) + (740712 + 183764 n)t - (196351 + 51729 n)t^2 + t^3.
\end{align*}
The discriminant of the cubic polynomial $q$ is negative for all $n \geq 0$, which implies that it has one real root and two non-real conjugate roots. Since $s$ is real, we obtain that $u$ and $v$ are non-real and satisfy $u = \bar{v}$. Therefore, $\beta$ is non-real. Moreover, if $|\beta| = 1$, then $\beta^{-1} = \bar{\beta}$ and $u = \beta + \bar{\beta}$ would be real. We obtain (without loss of generality) that $|\beta| > 1$. Finally, $p$ is an irreducible polynomial if $n \mod 3 = 2$. Indeed, its modulus-three reduction is $1 + 2 t + t^2 + 2 t^3 + t^4 + 2 t^5 + t^6$ in such case, which is readily seen to be irreducible over $\mathbb{F}_3$. We obtain that $\beta$ and $\alpha^{-1}$ are Galois-conjugates. By \cite[Lemma 7.9]{CGM}, we conclude that $\beta/|\beta|$ is not a root of unity.

\section*{Appendix}

\begingroup
\setcounter{thm}{0}
\renewcommand\thethm{A.\arabic{thm}}

In this section we detail the computation of the set $\Psi$ for the Arnoux--Yoccoz fractals. These fractals are extensively studied in Section 8 of \cite{CGM}.

For any $a \in \A$, $(p,c,s) \in \bA_a$ and $\tau \in \SS^1$ it is possible to compute $v_{a,(p,c,s)}^{(n)}(\tau)$ numerically by using a dynamic programming approach. This fact and the next lemma allow to compute the first coordinates of the representation of an extreme point.

\begin{lem}\label{lem:computation}
Let $a \in \A$, $\tau \in \SS^1$ and $(p,c,s), (\bar p, \bar c, \bar s) \in \bA_a$. There exists $C > 0$ such that for any $n \geq 1$ one has that
$$v_{a,(p,c,s)}(\tau) - v_{a,(\bar p, \bar c, \bar s)}(\tau) \leq v_{a,(p,c,s)}^{(n)}(\tau) - v_{a,(\bar p, \bar c, \bar s)}^{(n)}(\tau) + C|\beta|^{-n}.$$
Therefore, if $v_{a,(p,c,s)}^{(n)}(\tau) < v_{a,(\bar p, \bar c, \bar s)}^{(n)}(\tau) - C|\beta|^{-n}$, then $v_{a,(p,c,s)}(\tau) < v_{a,(\bar p, \bar c, \bar s)}(\tau)$.
\end{lem}

\begin{proof}
We have that $v_{a,(p,c,s)}(\tau) \leq v_{a,(p,c,s)}^{(n)}(\tau)$. Moreover, by Lemma \ref{lem:converg-exponencial}, we have that $v^{(n)}_{a,(\bar p,\bar c,\bar s)}(\tau) - v_{a,(\bar p,\bar c,\bar s)}(\tau) \le C |\beta|^{-n}$ for some $C > 0$. By adding both inequalities we obtain the desired result.
\end{proof}

The optimal constant $C > 0$ of the previous lemma is $\max\{-v_a(\tau); a \in \A, \tau \in \SS^1\}$. Any larger constant is valid as well, so we may choose any $C > 0$ such that $|z| \leq C$ for every $z \in \F_a$ and $a \in \A$. A simple choice is $C = \sum_{m \geq 1} |\beta|^{-m} |\Gamma(p)|$, where $(p,c,s) \in \bA$ is chosen so that $|\Gamma(p)| \geq |\Gamma(\bar p)|$ for every $(\bar p, \bar c, \bar s) \in \bA$.

For the case of the Arnoux--Yoccoz fractals, the prefix $p = 2$ satisfies the previous condition, so $|\Gamma(p)| = |\beta|$ and we choose $C = \frac{|\beta|}{|\beta| - 1} \approx 3.807$.

The strategy to compute $\Psi$ is the following: first notice that, since $\sigma(8) = 2$, one has that $\Psi_8 = \varnothing$. We will then assume that $a \in \{1, 2, 4, 5, 7\}$. Lemma \ref{lem:computation} allows us to know in which subfractal is the minimum attained. By using a binary search approach, we can obtain sufficiently good bounds for exactly two distinct directions in $\Psi_a$. The next lemma shows that these are the only elements of $\Psi_a$.

\begin{lem}
	Assume that $\sigma(a) = bc$. One has that $|\Psi_a| = 2$ for every $a \in \A$.
\end{lem}

\begin{proof}
	By the previous discussion, $|\Psi_a| \geq 2$. We will show that $|\Psi_a| \leq 2$.

	By Corollary 8.7 and Lemma 8.8 in \cite{CGM}, we know that $T$ has the u.r.p.\ for $\beta$ and that each $\F_a$ is the closure of the Jordan interior of a Jordan curve $\CC_a$.
	
	Since $\sigma(a) = bc$, we have that $\F_a = \F_{a,(\varepsilon,b,c)} \cup \F_{a,(b,c,\varepsilon)}$. By Lemma 8.6 in \cite{CGM}, we have that the interiors of $\F_{a,(\varepsilon,b,c)}$ and $\F_{a,(b,c,\varepsilon)}$ are disjoint. Assume by contradiction that $\tau_1, \tau_2, \tau_3 \in \SS^1$ are distinct elements of $\Psi_a$. Let $z_1, z_2, z_3 \in \F_{a,(\varepsilon,b,c)}$ and $z_1', z_2', z_3' \in \F_{a,(b,c,\varepsilon)}$ be extreme points for the directions $\tau_1, \tau_2$ and $\tau_3$, respectively. By the u.r.p., $z_1 \neq z_1'$, $z_2 \neq z_2'$ and $z_3 \neq z_3'$.
	
	Since $\F_{a,(\varepsilon,b,c)}$ is the closure of the Jordan interior of a Jordan curve, it is homeomorphic to a closed disc, so there exists a curve $\kappa\colon [0, 1] \to \F_{a,(\varepsilon,b,c)}$ such that $\kappa(0) = z_1$, $\kappa(1/2) = z_2$, $\kappa(1) = z_3$ and $\kappa(t)$ lies in the interior of $\F_{a,(\varepsilon,b,c)}$ for every $t \notin \{0, 1/2, 1\}$. We have that $\kappa(t) \notin\F_{a,(b,c,\varepsilon)}$ for every $t \in [0, 1]$.
	
Let $\Delta$ be the unique 2-simplex with $z_1, z_1', z_2, z_2', z_3, z_3' \in \partial\Delta$. Note that it is not possible that $z_1 = z_2 = z_3$ or $z_1' = z_2' = z_3'$, so $\Delta$ is indeed a non-degenerate 2-simplex. By definition, one has that $\kappa(t) \notin \partial\Delta$ if $\kappa(t) \notin \{0, 1/2, 1\}$. Therefore, $\Delta \setminus \kappa([0, 1])$ has at least two arc-connected components, one of which contains $\F_{a,(b,c,\varepsilon)}$. Each connected component intersects at most two of the three line segments in $\partial \Delta$, which is a contradiction since $\F_{a,(b,c,\varepsilon)}$ intersects the three lines.
\end{proof}

Lemma \ref{lem:computation} then allows to compute the first coordinates of the extreme points for the upper and lower bound for the directions in $\Psi_a$. We observe that, in most cases, these coordinates are equal after a few steps, even if they start in different subfractals. This fact produces an equation for some elements of $\Psi_a$.

For the other cases, the coordinates do not appear to become equal after any numbers of steps. For these directions we can only obtain bounds and we are not able to compute the exact values. Nevertheless, the bounds are good enough to compute the exact minimal components of $H$.
\bigbreak
\endgroup

\textbf{Acknowledgement:} The first author is grateful to the MathAmsud grant DCS-2017.
The second and third authors are grateful to CMM-Basal grant PFB-03.

\sloppy\printbibliography

\end{document}